\documentclass[11pt]{amsart}
\usepackage[left=2cm,top=2cm,right=3cm,nofoot]{geometry}
\usepackage{amsmath,mathtools}%
\usepackage{amsfonts}%
\usepackage{amssymb}%
\usepackage{graphicx}
\usepackage{esint}
\usepackage{hyperref}
\usepackage{enumitem}
\usepackage{accents}
\usepackage{etoolbox}
\patchcmd{\subsection}{-.5em}{.5em}{}{}
\newtheorem{theorem}{Theorem}[section]
\theoremstyle{plain}

\newtheorem{corollary}[theorem]{Corollary}

\newtheorem{definition}[theorem]{Definition}

\newtheorem{lemma}[theorem]{Lemma}

\newtheorem{proposition}[theorem]{Proposition}
\newtheorem{remark}[theorem]{Remark}

\numberwithin{equation}{section}
\theoremstyle{plain}

\usepackage{etoolbox}
\AtEndEnvironment{proof}{\setcounter{claim}{0}}

\newcommand{\diag}{\operatorname{diag}}
\newcommand{\Ad}{\operatorname{Ad}}
\newcommand{\tr}{\operatorname{tr}}

\newcommand{\hpn}{{\bf HP}^n}

\newcommand{\re}{\mathbb{R}}

\begin{document}
\title[Global bifurcation  for Yamabe type equations]{Global bifurcation techniques for Yamabe type equations on Riemannian manifolds}

\author{Alejandro Betancourt de la Parra}\thanks{The authors were supported by grant 220074 of Fondo Sectorial de Investigaci\'{o}n para la Educaci\'{o}n SEP-CONACYT}
\address{Centro de Investigaci\'{o}n en Matem\'{a}ticas, CIMAT, Calle Jalisco s/n, 36023 Guanajuato, Guanajuato, M\'{e}xico}
\email{alejandro.betancourt@cimat.mx}

\author{Jurgen Julio-Batalla}
\address{Centro de Investigaci\'{o}n en Matem\'{a}ticas, CIMAT, Calle Jalisco s/n, 36023 Guanajuato, Guanajuato, M\'{e}xico}
\email{jurgen.julio@cimat.mx}

\author{Jimmy Petean}
\address{Centro de Investigaci\'{o}n en Matem\'{a}ticas, CIMAT, Calle Jalisco s/n, 36023 Guanajuato, Guanajuato, M\'{e}xico}
\email{jimmy@cimat.mx}

\begin{abstract} We consider a closed Riemannian manifold $(M^n ,g)$ of dimension $n\geq 3$ 
and study positive solutions of the equation  $-\Delta_g u + \lambda u = \lambda u^q$, with $\lambda >0$, $q>1$.
If $M$ supports
a proper isoparametric function with focal varieties $M_1$, $M_2$ of dimension $d_1 \geq d_2 $
we show that for any $q<\frac{ n-d_2+2 }{n- d_2 -2}$  the number of
positive solutions of the equation $-\Delta_g u + \lambda u = \lambda u^q$ tends to $\infty$ as $\lambda \rightarrow +\infty$. 
We apply this result to prove multiplicity results for positive solutions of critical and supercritical equations. In particular
we prove multiplicity results for the Yamabe equation on Riemannian manifolds.

\end{abstract}
\maketitle

\section{Introduction}

On a closed Riemannian manifold $(M^n ,g)$ of dimension $n\geq 3$ we
consider the following {\it Yamabe type} equation
\begin{equation}\label{equation}
-\Delta_g u+ \lambda u= \lambda u^{q},
\end{equation}

\noindent
with $\lambda >0$ and $q>1$. Let $p_n=\frac{n+2}{n-2}$. Equation (\ref{equation}) is said to be critical if $q=p_n$, 
subcritical if $q<p_n $ and supercritical if $q>p_n$. Note that equation (\ref{equation}) is equivalent 
to the equation $-\Delta_g u + \lambda u = \mu u^q$, for any positive constant $\mu$. We pick to normalize the
equation in this way so that $u=1$ is a solution for any $\lambda$. It is also very common to write the equation
as $-\Delta_g u + \lambda u = u^q$, 
in this case the constant solution is $u = \lambda^{\frac{1}{q-1}}$. The equation is very well-known and has been extensively studied 
in the last decades, 
both on Euclidean space (see for instance \cite{ BenciCerami,BrezisNirenberg, NiWei}) and on  Riemannian manifolds \cite{Gidas, Yanyanli, MichelettiPistoia}.   The critical case appears in Riemannian geometry when trying to solve the problem of
finding metrics of constant scalar curvature in a given conformal class of metrics; what is known as the Yamabe problem. 
Namely, if a metric $h$ conformal to $g$ is expressed as $h=u^{p_n -1} g$ for a positive function $u$ on $M$, then the 
scalar curvature of $h$, $s_h$, is equal to a constant $\lambda \in {\bf R}$ if and only if $u$ solves 
the Yamabe equation:

\begin{equation}\label{YamabeEquation}
- a_n \Delta_g u+ s_g u= \lambda u^{p_n},
\end{equation}

\noindent
where $a_n = \frac{4(n-1)}{n-2}$. A fundamental result obtained in several steps by H. Yamabe \cite{Yamabe}, N. Trudinger
\cite{Trudinger}, T. Aubin \cite{Aubin} and R. Schoen \cite{Schoen0} states  that there is always at least one positive solution
of the Yamabe equation, for any closed Riemannian manifold. Since the Yamabe equation is conformally invariant,  
in order to understand the space of
solutions we can therefore assume that $s_g$ is constant (and positive. In the non-positive case (\ref{YamabeEquation})
only has the constant solution), and (\ref{YamabeEquation}) is of the form 
(\ref{equation}) with $q=p_n$ and $\lambda =\frac{s_g}{a_n}$.

\vspace{.5cm}

In this article we will assume that there is an isoparametric function $f:M \rightarrow \re$ and look for solutions of
equation (\ref{equation}) of the form $v=\varphi \circ f$, where $\varphi :\re \rightarrow \re_{>0}$. 
We will call functions $v$ of this form {\it $f$-invariant}. 
Recall that a   function $f:M\rightarrow[t_0,t_1]$  is called isoparametric if  $|\nabla f|^2=b(f)$, $\Delta f=a(f)$ for 
some smooth functions $a,b$. Isoparametric functions on general Riemannian manifolds were considered in
\cite{Wang}, following the classical work by E. Cartan \cite{Cartan}, B. Segre \cite{Segre}, T. Levi-Civita \cite{Civita}  in the case of space forms. 
It is known from the general theory of isoparametric functions  that the only zeros of the function  $b:[t_0,t_1]\rightarrow \mathbb{R}_{\geq 0}$ are $t_0$ and $t_1$. Moreover $M_1 = f^{-1} (t_0 )$ and
$M_2 = f^{-1} (t_1 )$ are smooth submanifolds and are called the {\it focal submanifolds}
of $f$. We denote by $d_i$ the dimension of $M_i$. If $d_1, d_2 \leq n-2$ we call $f$ a {\it proper} isoparametric function, as in \cite{Tang}.
We will assume that $f$ is proper. 
The most familiar case of isoparametric functions comes from cohomogeneity
one isometric actions. Assume that $G$ acts isometrically on $(M,g)$ with regular orbits of 
codimension one and that the orbit space is  an interval. If $f$ is a smooth function which is $G$-invariant and
its only critical points are the two singular orbits, then $f$ is isoparametric, and a function is $f$-invariant
if and only if it is $G$-invariant.  Note that in this situation the singular orbits have codimension at least 2, and therefore
the isoparametric function $f$ is proper.  But there are proper isoparametric functions  in
much more general situations than the one of cohomogeneity one isometric actions. For instance in \cite{QianTang} C. Qian and
Z. Tang proved that given a Morse-Bott function $f$ on a closed manifold $M$ (with appropriate conditions on its critical set) there is a
Riemannian metric $g$ on $M$ so that $f$ is isoparametric for $(M,g)$.  In \cite{Henry} this 
situation was considered when $(M,g)$ is the round sphere $(\mathbb{S}^n , g_0^n )$ and $q$ is subcritical, and multiplicity results
were obtained in this case for $f$-invariant solutions of equation (\ref{equation}).

\vspace{.5cm}

We consider the space $C_f^{2,\alpha}$ of $C^{2,\alpha}$ functions on $M$ which are $f$-invariant and
we consider equation (\ref{equation}) as an operator equation on $(u,\lambda ) \in C_f^{2,\alpha} \times (0,\infty )$. We study
solutions bifurcating from the family of trivial solutions $\lambda \mapsto (1, \lambda)$. One can see (see \cite{Henry}) that there
are infinite eigenvalues of $-\Delta_g$ which have non-trivial eigenfuncions in $C_f^{2,\alpha}$. 
Using the well-known theory of local 
bifurcation for simple eigenvalues \cite{Crandall} we can easily see that:

\begin{theorem}\label{LocalBifurcationTheorem} Let $\mu_k$, $k\geq 1$, be the positive eigenvalues of
$-\Delta_g |_{C_f^{2,\alpha}}$. For any $q>1$ let $\lambda_k = \lambda_k (q) = \frac{\mu_k}{q-1}$. There are  branches
$t \mapsto (u(t), \lambda (t) )$, $t\in (-\varepsilon , \varepsilon )$, of $f$-invariant solutions of (\ref{equation}) so that
$\lambda (0) =\lambda_k$, $u(0)=1$ and $u(t) \neq 1$ if $t\neq 0$.
\end{theorem}

Next we study the global behavior of the local branches appearing at the bifurcation points $(1,\lambda_k )$. We will apply the global bifurcation theorem of
P. Rabinowitz \cite{Rabinowitz}. To do so we will need to impose conditions on $q$. Recall that  $d_i$ is the dimension of the focal submanifold $M_i$ of the proper isoparametric function  and let $d=\min \{d_1 , d_2 \} \leq n-2$. Then we let
$p_f = \frac{n-d +2}{n-d-2}$, $p_f = \infty$ in case $d=n-2$. Note that if $d>0$ then $p_f > p_n$. For the
next results we will ask that $q< p_f$. If $d>0$ the results apply then to supercritical equations. 

An interesting  question that was raised for instance in \cite{Veron, Brezis, Licois} is to find conditions under which for $\lambda $ small the only
solution of (\ref{equation}) is the trivial solution $u=1$.
In fact, in the subcritical case J. R. Licois and L. Veron proved in \cite{Licois} that there exists some positive $\lambda_0=\lambda_0(M,g,q)$ for which the equation (\ref{equation}) admits only the constant solution for all $0<\lambda<\lambda_0$.  We will prove a similar result, restricting to $f$-invariant solutions but allowing $q$ to
be supercritical:

\begin{theorem}\label{SmallLambda} If $q<p_f$ there exists $\lambda_0 >0$ such that if $\lambda \in (0, \lambda_0 )$ and
$u$ is a positive $f$-invariant solution of (\ref{equation}) then $u=1$.
\end{theorem}
 
Theorem \ref{LocalBifurcationTheorem} says that at each bifurcation point $(1,\lambda_k )$ appears a branch $B_k$ of 
nontrivial solutions. Explicitly, we let $B_k$ be the connected component containing the  nontrivial solutions appearing close
to $(1, \lambda_k )$, in the space of nontrivial solutions of  (\ref{equation}). Theorem \ref{SmallLambda} says in particular that if $(u,\lambda ) \in B_k$ then $\lambda \in [\lambda_0 , \infty )$.
This will allow us to apply
the global bifurcation theorem to prove:

\begin{theorem}\label{GlobalBifurcationTheorem}Let $(M^n ,g)$ be a closed Riemannian manifold of dimension $n\geq 3$,
and  $f:M \rightarrow \re$ be a proper isoparametric function. For any $q \in (1, p_f  )$ let
$\lambda_k (q)= \frac{\mu_k}{q-1}$, as in Theorem (\ref{LocalBifurcationTheorem}). Then 
for any positive integer $k$  and for any $\lambda \in
(\lambda_k , \lambda_{k+1} ]$ equation (\ref{equation}) has at least $k$ different positive 
$f$-invariant solutions. 

\end{theorem}

The theorem is proved by showing that the branches $B_k$ 
are disjoint to each other and  for any $\lambda > \lambda_k$ there exists a solution $(u,\lambda) \in B_k$. We will prove in Section 4
(Corollary 4.6) that  for any fixed $\lambda >0$ only a finite number of the branches cut the ``vertical'' line $C_f^{2,\alpha} \times
\{ \lambda \}$. This implies that for any $K>0$ there exists $k_0$ such that if $k\geq k_0$ the branch $B_k \subset C_f^{2,\alpha}
\times [K,\infty )$.

\vspace{.5cm}

Consider the isometric  $O(n)$-action on the curvature one metric on the sphere, $(\mathbb{ S}^n,g_0)$, fixing an axis.  A linear 
function invariant by the action gives a proper  isoparametric function. In this case $d=0$ 
(the fixed points of the action are two points) and Theorem \ref{GlobalBifurcationTheorem}
applies to the subcritical case $q < p_n$. In this case the Theorem  \ref{GlobalBifurcationTheorem} was proved by 
Q. Jin, YY. Li and H. Xu in \cite{Yanyanli}  (note that in this case the invariant functions are precisely the  radial functions,
with respect to the invariant  axis). In this case Theorem \ref{SmallLambda} was proved by M-F. Bidaut-Veron and L. Veron in
\cite{Veron}, and the constant $\lambda_0$ is explicit: $\lambda_0 = \frac{n}{q-1}$. In \cite{Henry} the authors  considered the case 
of any isoparametric function on the sphere, but again only for the subcritical case. 

\vspace{.5cm}

We will now consider some applications of Theorem \ref{GlobalBifurcationTheorem}. 
The simplest situation to apply Theorem \ref{GlobalBifurcationTheorem}, beyond the case of radial functions mentioned
above, is to consider an isoparametric function $f$ on $(\mathbb{ S}^3, g_0 )$  invariant by the natural isometric cohomogeneity one action of $\mathbb{S}^1 \times \mathbb{S}^1$. Both singular orbits have dimension 1, so $f$ is proper and $p_f = \infty$. Moreover the values of $\lambda_k$ in Theorem \ref{LocalBifurcationTheorem} and Theorem \ref{GlobalBifurcationTheorem} are easily seen to be $\lambda_k = \frac{\mu_{2k}}{q-1}$, where $\mu_k$ are the eigenvalues of $-\Delta_{g_0}$ (it is well-known that $\mu_k = k(k+2)$). Therefore Theorem  \ref{GlobalBifurcationTheorem} says

\begin{corollary} For any $q>1$ the equation (\ref{equation}) on $(\mathbb{ S}^3, g_0 )$ has at least 
$k$ positive different torus invariant solutions if $\lambda \in \left( \frac{4k(k+1)}{q-1} , \frac{4(k+1)(k+2)}{q-1} \right]$.
\end{corollary}

Note that as $q\rightarrow \infty$ we obtain solutions with $\lambda$ very close to 0.

In $(\mathbb{ S}^4, g_0^4 )$ one could
consider an isoparametric function $f$ invariant by the isometric cohomogeneity one action of $O(3) \times O(2)$. In this
case the singular orbits have dimension 1 and 2, respectively.  Then $f$ is proper and $p_f = 5$ (note that 
$p_4 =3$). One can easily see also that the eigenvalues of $-\Delta_{g_0}$
restricted to  $O(3) \times O(2)$-invariant functions are $\mu_{2k}$, where $\mu_k = k(k+3)$ are the 
eigenvalues of $-\Delta_{g_0}$. Therefore we obtain:

\begin{corollary} For any $q\in (1,5)$ the equation (\ref{equation}) on $(\mathbb{ S}^4, g_0 )$ has at least 
$k$ positive different $O(3) \times O(2)$-invariant solutions if $\lambda \in \left( \frac{2k(2k+3)}{q-1} , \frac{2(k+1)(2k+5)}{q-1} 
\right]$.
\end{corollary}

Similar applications can be obtained given any isoparametric functions on a sphere. To apply Theorem \ref{GlobalBifurcationTheorem}
one always needs to compute the dimension of the focal submanifolds (which is usually simple) and the eigenvalues of the restricted
Laplacian. We will show how to do these computations explicitly in Section 7.

\vspace{.5cm}

Let us now consider some more general spaces. For instance one has isoparametric functions invariant by cohomogeneity
one actions on the complex projective spaces $({\bf CP}^n ,g_{FS} )$, where $g_{FS}$ is the Fubini-Study metric. The simplest is the action by $U(n)$ for which the singular orbits are 
a point and ${\bf CP}^{n-1}$. Theorem \ref{GlobalBifurcationTheorem} then gives solutions only in the subcritical case. But one
can consider other cohomogeneity one actions. For instance 
in the case of ${\bf CP}^2$ there is a natural cohomogeneity one action by $SO(3)$ (obtained by considering the inclusion 
$SO(3) \subset U(3)$). See for instance the discussion in \cite[Page 238]{Ziller2}, or in \cite[Section 3]{BettiolKrishnan}) which has singular orbits of
dimension 2: the real points ${\bf RP}^2 \subset {\bf CP}^2$ and $\{ [z_0 , z_1 , z_2 ] \in {\bf CP}^2 : z_0^2 + z_1^2 + z_2^2 =0 \}
=\mathbb{S}^2$. Then an invariant isoparametric function $f$ is proper  and $p_f = \infty$. It is well known that 
the eigenvalues of $-\Delta_{g_{FS}}$ are $\mu_k = 4k(k +2)$. 
One can see that the eigenvalues
of $-\Delta_{{\bf CP}^2}$ restricted to $SO(3)$-invariant functions are $\mu_{2k} = 16k(k+1)$: we will show how to prove this and similar eigenvalues computations in Section 7. Therefore Theorem \ref{GlobalBifurcationTheorem} says

\begin{corollary} For any $q>1$ the equation (\ref{equation}) on $({\bf CP}^2, g_{FS})$ has at least 
$k$ positive different $SO(3)$-invariant solutions if $\lambda \in \left( \frac{16k(k+1)}{q-1} , \frac{16(k+1)(k+2)}{q-1} \right]$.
\end{corollary}

The example generalizes to higher dimensional complex projective space (see Corollary \ref{7} in section 7). 

There is also an action on $({\bf CP}^n ,g_{FS} )$
by $U(k) \times U(l)$ where $k+l = n+1$ and $k \geq l \geq 2$. We will see in section 7 that  for an 
invariant isoparametric function $f$ one has $p_f = \frac{2n-2l +4}{2n-2l}>p_{2n}$ and that the $f$-invariant 
eigenvalues are the same as the eigenvalues of the full Laplacian, then applying Theorem \ref{GlobalBifurcationTheorem}  we obtain:

\begin{corollary} For any $q\in (1, \frac{2n-2l +4}{2n-2l} )$ the equation (\ref{equation}) on $({\bf CP}^n, g_{FS})$ has at least 
$k$ positive different $U(k)\times U(l)$-invariant solutions if $\lambda \in \left( \frac{4k(k+2)}{q-1} , \frac{4(k+1)(k+3)}{q-1} \right]$.
\end{corollary}

\vspace{.5cm}

We consider now the quaternionic projective spaces ${\bf HP}^n$. It also has a canonical metric  analogous to the Fubini-Study metric for the complex projective space. We denote this metric by $h_{FS}$. It is  a homogeneous Einstein metric with positive scalar curvature.  With the canonical  metrics  the usual Hopf fibration
$$ \mathbb{S}^{3}\rightarrow \mathbb{S}^{4n+3} \rightarrow {\bf HP}^n $$
is a Riemannian submersion with totally geodesic fibers. 
${\bf HP}^1 = \mathbb{S}^4$, so we will consider now the case $n\geq 2$. The simplest cohomogeneity one action on ${\bf HP}^n$ is the action by $Sp(n)$ which
has singular orbits a point and ${\bf HP}^{n-1}$. Then an invariant isoparametric function $f$ is proper and $p_f =p_{4n}$, and
Theorem \ref{GlobalBifurcationTheorem} applies to subcritical equations. But more generally there is also  the cohomogeneity one action by 
$Sp(k) \times Sp(l)$, with $k +l =n+1$, $k,l \geq 2$; for this action the singular orbits are
${\bf HP}^{k -1}$ and ${\bf HP}^{l -1}$. Then if $k \geq l$ we have that an invariant isoparametric function
$f$ is proper and $p_f =  \frac{4n-4l+6}{4n-4l+2} >p_{4n}$. In these two cases
the  $f$-invariant eigenvalues of $-\Delta_{h_{FS}}$ are
$\lambda_{2k} =2k(2k+4n+2) $, which are the same as the eigenvalues of the full Laplacian.
There is also a cohomogeneity one action by $U(n+1)$. For an isoparametric function $f$ invariant by this
$U(n+1)$-action one can see that 
the $f$-invariant eigenvalues are $\lambda_{4k}=4k(4k+4n+2)$ and   $p_f =\frac{n+1}{n-1}>p_{4n}$. The details 
of these calculations will
be given in section 7. All the previous examples come from cohomogeneity one actions but there are also
non-homogeneous examples. We will discuss these examples in section 7, but we do not include them here.  
Note that in the case of the actions of $Sp(k) \times Sp(l)$ and $U(n+1)$ Theorem \ref{GlobalBifurcationTheorem} applies to supercritical equations, but for the sake of simplicity we will consider in the next corollary only
the subcritical case $q<p_{4n}$. 

Let 

$$\alpha_k^q  = \frac{4k(4k+4n+2)}{q-1} .$$

It follows from the previous discussion and Theorem \ref{GlobalBifurcationTheorem} that as $\lambda$ crosses $\alpha_k^q$ one has: two new solutions which are  $Sp(n)$-invariant, two new solutions which are 
$Sp(k) \times Sp(l)$-invariant (for each $l$, $2\leq l  \leq [(n+1)/2]$, and $k+l =n+1$) and one new solution which is $U(n+1)$-invariant.
Therefore we have:

\begin{corollary} Let $q\in (1, p_{4n} )$ for any $\lambda \in (\alpha_k^q , \alpha_{k+1}^q ]$ 
equation (\ref{equation}) on  $({\bf HP}^n , h_{FS})$ has at least $(2[(n+1)/2]  +1) k$ different solutions.
\end{corollary} 

\vspace{1cm}

Finally we will apply the previous examples  to obtain multiplicity results for the Yamabe equation on 
homogeneous metrics on spheres. 
Many results on multiplicity of solutions
of the Yamabe equation have been obtained in recent years. 
The first nontrivial space of solutions was studied by R. Schoen in \cite{Schoen} (see also the article by O. Kobayashi
\cite{Kobayashi}): the Riemannian product of $(\mathbb{ S}^n, g_0 )$ with a circle of radius $T>0$. 
Schoen pointed out that all solutions depend only on the $\mathbb{S}^1$-variable and therefore the Yamabe equation reduces 
to an ordinary differential equation, and by studying this ODE showed that the number of solutions grows at
$T \rightarrow \infty$.  ODE techniques were also used in \cite{Petean} to consider the case when $S^1$ is 
replaced by $(\mathbb{ S}^m, g_0 )$ with $m\geq 2$, proving similar multiplicity results looking for solution which
only depend on $(\mathbb{ S}^m, g_0 )$  and are radial.
Other multiplicity results were obtained using bifurcation theory. To apply bifurcation theory one considers a one dimensional family of 
constant scalar curvature Riemannian metrics (on a closed manifold) and look for the elements in the family where new nearby solutions of the
Yamabe equation appear (these elements are called the {\it bifurcation points} of the family). The general frame to apply 
local bifurcation theory to the Yamabe equation was studied in \cite{deLima}: as a consequence the authors  proved for instance
that in the family Riemannian products of constant positive scalar curvature were one of the factors is multiplied by
$t>0$ there exist infinite bifurcation points. See also \cite{Bettiol} for related results. Global bifurcation theory 
(which considers the global structure of the space of solutions appearing at the bifurcation points) was only applied 
in the case of a Riemannian product with a sphere (see \cite{Henry,Yanyanli}). Theorem \ref{GlobalBifurcationTheorem}
allows us to obtain multiplicity results in much more general situations. 
As an example we will focus on  the family of homogeneous metrics on spheres. This case was considered by R. Bettiol and P. Piccione  in
\cite{BettiolPiccione} where the authors show the existence of infinite such homogeneous metrics which are 
bifurcation points (in an appropriate setting) for the Yamabe equation (see  also  \cite{Otoba} for a related discussion).  We will give a brief
discussion about 
homogeneous metris on spheres in the appendix at the end of the article, summarizing the description given for instance
in {\cite{BettiolPiccione, ziller1982}. There are three families of homogeneous metrics on spheres. In the
three families the homogeneous metrics appear as the metrics in the total space of a Riemannian submersion with totally geodesic
fibers. One family is obtained by considering the isometric $\mathbb{S}^1$-action on $(\mathbb{ S}^{2n+1} , g_0 )$. The quotient is 
$({\bf CP}^n ,g_{FS} )$. Multiplying by $x>0$ the metric in the direction of the orbits one obtains a one-dimensional family of
homogeneous metrics $g_x$  on $\mathbb{ S}^{2n+1}$, with a fibration

$$(\mathbb{ S}^1, x\, \sigma^2) \rightarrow (\mathbb{ S}^{2n+1} , g_x ) \rightarrow ({\bf CP}^n ,g_{FS} )$$

\noindent
with totally geodesic fibers (here $\sigma^2$ is the standard metric on $S^1$). The scalar curvature of  $g_x$ is 
$$  s_{g_x} = 4n^2 + 4n-2nx.$$
(see \cite{BettiolPiccione} or the discussion in the appendix). It
was shown in \cite{BettiolPiccione} that this family is locally rigid, i.e. there are no other solutions of the Yamabe equations close to
the elements of the family. 

The second family (and the most interesting one from the point of view of bifurcation for solutions of the Yamabe equation) is obtained by considering the  Hopf fibration

$$(  \mathbb{S}^{3}  , \Sigma_{i=1}^3 x_i (\sigma^i)^2) )\rightarrow ( \mathbb{S}^{4n+3} , g_x ) \rightarrow ( {\bf HP}^n , h_{FS} ),$$
where $x=(x_1,x_2,x_3)$, $x_i >0$. The metric on the fiber $\mathbb{S}^{3}$ is any homogeneous metric, and there is a
 three parameters family of them (details can be found in the appendix).  The scalar curvature of  $g_x$ is 
$$s_{g_x}=\frac{2}{x_1 x_2 x_3}\left(x_1^2 +x_2^2 +x_3^2-(x_2-x_3)^2-(x_3-x_1)^2-(x_1-x_2)^2\right)-4n(x_1+x_2+x_3)+16n^2+32n.$$ 
Therefore the Yamabe equation on 
$ ( \mathbb{S}^{4n+3} , g_x )$ can be written as:

$$-\Delta_{  g_{x}  } u + \frac{s_{g_x}}{a_{4n+3}} u  =  \frac{s_{g_x}}{a_{4n+3}} u^{p_{4n+3}} ,$$

\noindent
which is equation (\ref{equation}) with $\lambda =  \frac{s_{g_x}}{a_{4n+3}}$.
We consider solutions of the form $u:  {\bf HP}^n \rightarrow \re_{>0}$. 
Since the projection $\pi:  ( \mathbb{S}^{4n+3} , g_x ) \rightarrow ( {\bf HP}^n , h_{FS} )$ is a Riemannian submersion
with totally geodesic fibers it follows that $\pi \circ \Delta_{g_x} = \Delta_{h_{FS}} \circ \pi$ and therefore $u$ is a
solution of equation (\ref{equation}) if and only if 

$$-\Delta_{  h_{FS}   } u + \frac{s_{g_x}}{a_{4n+3}} u =  \frac{s_{g_x}}{a_{4n+3}} u^{p_{4n+3}} $$

Then one can apply  Theorem \ref{GlobalBifurcationTheorem}, or more explicitly Corollary 1.8, to produce solutions of the Yamabe 
equation in terms of $s_{g_x}$. We see that as $s_{g_x}$ goes to infinity the number of solutions provided by Corollary 1.8 goes to infinity. Explicitly we obtain:

\begin{corollary} If $s_{g_x} \in (a_{4n+3}  \ \alpha_k^{p_{4n+3}} ,  a_{4n+3}  \ \alpha_{k+1}^{p_{4n+3}} ]$ then
the Yamabe equation on  $( \mathbb{S}^{4n+3} , g_x )$ has at least  $2([(n+1)/2]  +1) k$ different solutions.

\end{corollary}

In  \cite{BettiolPiccione} the authors consider the particular case when $x_1 , x_2, x_3$ are equal and therefore the 
metric on $\mathbb{S}^3$ is a multiple of the round metric $g_0$. In \cite[Proposition 8.2]{BettiolPiccione} the authors proved that there are infinitely many bifurcation points, and close to the points there are homogeneous metrics for which the Yamabe equation has at least 3 different solutions. 
In \cite{Otoba} the authors prove that all the solutions appearing in these bifurcation points are actually functions of ${\bf HP}^n$, as in the previous Corollary.

\vspace{.5cm}

The third family of homogeneous metrics on spheres appears only in $ \mathbb{S}^{15}$. One has a Riemannian fibration
$\mathbb{S}^7 \rightarrow \mathbb{S}^{15} \rightarrow \mathbb{S}^8$ and can apply the same ideas by looking at
isoparametric functions on $ \mathbb{S}^8$. 

\vspace{.5cm}

{\bf Questions}: The ultimate goal would be to understand all solutions of the Yamabe equation on homogeneous metrics
on spheres. The first question would be if the space of solutions $\{ (g, u ) : $ $g$ is a homogeneous metric on
$ \mathbb{S}^{n}$ and $u$ is a positive solution of the Yamabe equation on $( \mathbb{S}^{n} ,g) \}$ is connected:
i.e. do all solutions belong to the connected component of the homogeneous metrics? Another interesting first step
would be to understand the structure of the space of solutions bifurcating from a bifurcation point. The results in
this article construct solutions which are $f$-invariant with respect to some isoparametric function $f$. Are all
bifurcating solutions of this form?

\vspace{.5cm}

The article is organized as follows. In section 2 we will write explicitly the ordinary differential equation given
by the restriction of the Yamabe equation to functions invariant by some isoparametric function. In section 3 we
will discuss the bifurcation points for the resulting ODE, proving Theorem 1.1. In section 4 we prove some technical lemmas that will
be used in the following sections. These results show that under the condition $q<p_f$ the ODE behaves
like a subcritical equation. In section 5 we will prove Theorem 1.2 and in section 6 we will prove Theorem 1.3. 
In section 7 we will describe how to compute the eigenvaules of the restricted Laplacian needed for the applications.
Finally in the appendix we will give a brief description the family of homogeneous metrics on spheres and in
particular give  formulas for the scalar curvature of such metrics.

\section{Yamabe-type equations for \texorpdfstring{$f$}{}-invariant functions}

\noindent
Let ($M^n,g$) be a closed connected Riemannian manifold of dimension $n$.
Recall that a function $f:M\rightarrow[t_0,t_1]$  is called  {\it isoparametric} if  $|\nabla f|^2=b(f)$, $\Delta f=a(f)$ for 
some smooth functions 
$a,b$. Isoparametric functions on general Riemannian manifolds were first considered by Q-M. Wang in
\cite{Wang}, following the classical theory of Cartan, Segre and others in the case of space forms.

From the general theory of isoparametric functions we know that the only zeros of the function  $b:[t_0,t_1]\rightarrow \mathbb{R}^+$ are $t_0$ and $t_1$,  which means that the only critical values of $f$ are its minimum and its
maximum (see \cite{Wang}). It is also proved in \cite{Wang} that $M_1 = f^{-1} (t_0 )$ and
$M_2 = f^{-1} (t_1 )$ are smooth submanifolds; they are called the {\it focal submanifolds}
of $f$. We let $d_i$ be  the dimension of $M_i$. There are examples, like the function $x_1^2$ on the round
$n$-sphere, of isoparametric functions such that  $d_1 =n-1$ (see \cite{Tang}).  But we will consider the generic case when
$d_1 , d_2 \leq n-2$. In this case the isoparametric function $f$ is called {\it proper}  and all the level sets
$M_t = f^{-1} (t)$ are connected (as proved in \cite[Proposition 2.1]{Tang}). It is also known that the
level sets $M_t$ are tubes around any of the focal submanifolds. 

Next we consider geodesics which are transversal to the level sets of $f$:

\begin{definition}
A geodesic $\gamma:[l_1,l_2]\rightarrow M$ is called an $f-$segment if $f(\gamma(l))$ is an increasing function of $l$ and $\gamma^{\prime}(l)=\nabla f/\sqrt{b}$ wherever $\nabla f\neq0$.
\end{definition}

Note that $f$-segments are parametrized by arc-length. It is also easy to see that the integral curves
of $\nabla f$ (parametrized by arc-length) are $f$-segments, and that $f$-segments realize the distance 
between the level sets $M_s$, $M_t$ (see \cite[Lemma 1]{Wang}). If $\gamma :[0,s] \rightarrow M$ is
an $f$-segment then $s= $ length$(\gamma ) =d(M_{f(\gamma (0))} , M_{f(\gamma s}))$, and $(f \circ \gamma )'(t) = \sqrt{b(f(\gamma (t))}$.
By reparametrizing $\gamma$ by $l=(f\circ\gamma)^{-1}(s)$ for $s\in[f(\gamma (0),f(\gamma (s))]$ it is easy to obtain the
formula for $d_g(M_c,M_d)$ for any $t_0 \leq c < d \leq t_1$:

\begin{equation*}
d_g(M_c,M_d)=\int_{c}^{d}\dfrac{1}{\sqrt{b(t)}} dt,
\end{equation*}

Let $t^*=d_g(M_{1},M_{2})$ and  ${\bf d} : M \rightarrow [0,t^* ]$,
${\bf d} (x) =d_g(M_1 ,x)$.

We will consider functions which are constant on the level sets of $f$:

\begin{definition}
A function $u:M\rightarrow \mathbb{R}$ is called $f$-invariant if $u(x)=\phi({\bf d}(x))$ for some  function $\phi:[0,t^* ]\rightarrow \mathbb{R}$.
\end{definition}

An $f$-invariant function $u$ is determined by the corresponding function $\phi$. Let us recall the most familiar case of
radial functions on Euclidean space. 
Namely we consider $F:B(0,\varepsilon ) \rightarrow
\re$ a radial function and write $F(x) = \alpha ( \| x \| )$ for some $\alpha :[0,\varepsilon ) \rightarrow \re$. 
Then it is well-known that if $F$ is
$C^1$ then $\alpha '(0) =0$. And one can write $\alpha (r)=F(r,0,...0)$ to easily see that if $F$ is of clase $C^{2,\alpha}$ then
also $\alpha \in C^{2,\alpha}$. The reciprocal is also well known: if $\alpha  \in C^{2,\alpha} [0,\varepsilon )$ and 
$\alpha '(0)=0$ then $F\in C^{2,\alpha} (B(0,\varepsilon ))$.
We have also the expresion for the Laplacian $\Delta F$:

\begin{equation}
\Delta F (x) = \alpha ''( \|x \| ) + \frac{n-1}{\| x \| } \alpha ' (\| x \| )
\end{equation}

\noindent
if $x\neq 0$, and taking the limit

\begin{equation}
\Delta F (0) = n \alpha ''(0) .
\end{equation}

We will consider the same identification for a general isoparametric function $f$.
We will denote by $\mathcal{B} = \{ \phi \in C^{2,\alpha} ([0,t^* ] ): \phi '(0) =0 = \phi '(t^* ) \}$.

\begin{lemma}
If we denote by
$C^{2,\alpha}_f (M)$ the set of $C^{2,\alpha}$ functions on $M$ which are $f$-invariant, then the
application $\phi \mapsto u(x)=\phi({\bf d}(x))$ identifies $\mathcal{B}$ with $C^{2,\alpha}_f (M)$.
\end{lemma}
\begin{proof}
Let $\phi :[0,t^* ] \rightarrow \re$ and $u(x)=\phi({\bf d}(x))$. Note  that away from the focal submanifolds ${\bf d}$ is smooth 
and has no critical points. Therefore 
it is clear that the function $u\in C^{2,\alpha}(f^{-1}(t_0 , t_1 ))$ if and only if $\phi \in C^{2,\alpha} (0,t^* )$. We only have to deal with 
the case $x\in M_{1}$ or $x\in M_{2}$. Both cases are equivalent so let us consider for instance $x\in M_{1}$. We choose normal coordinates  around $x$ in $M_{1}$ and then normal coordinates for the normal bundle of $M_{1}$. Namely, around $x$ we
put coordinates of the form $(z,y)$ where $(z_1,...z_{d_1} ) \in {\bf R}^{d_1}$, $y \in {\bf R}^{n-d_1}$ and $\|y \| = {\bf d} (z,y)$.
Therefore in these coordinates $u(z,y)=u(y) =  \phi (\| y \| )$. 

The lemma is then reduced to the familiar case of radial functions on Euclidean space mentioned above. 

\end{proof}

Now we obtain the expression of  the Yamabe-type equations of $(M,g)$ for $f$-invariant functions. For this one needs to
express $\Delta u$ in terms of $\phi$. 

If for $x$ is not in the focal submanifolds, we have
\begin{align*}
\Delta(u(x))&=\Delta(\phi({\bf d}(x))\\
&=\phi^{''}  (  {\bf d}(x) ) \    (|\nabla {\bf d}(x))|^2) + \phi'  ({\bf d}(x) ) \ (\Delta ({\bf d}(x))).
\end{align*}
But
\begin{align*}
|\nabla {\bf d}|^2&=1\quad\text{and},\\
\Delta {\bf d}(x)&=\Delta \left(\int_{t_0}^{f(x)}\dfrac{dt}{\sqrt{b(t)}} \right)=\frac{-b'}{2b\sqrt{b}} (f(x)) |\nabla f|^2 (x)+\frac{\Delta f (x)}{\sqrt{b} (f(x) )}\\
&=\frac{1}{2\sqrt{b}}(-b'+2a) (f(x) )= h({\bf d} (x) ),
\end{align*}
where we denote by $h(t)$  the mean curvature  of the hypersurface ${\bf d}^{-1} (t)$ (see for instance \cite[page 165]{Tang}).

The function $h$ is obviously smooth in $(0,t^* )$. It is also well-known the asymptotic behaviour of the mean curvature close to focal submanifolds. Since each ${\bf d}^{-1} (t)$ is a tube over either $M_{1}$ or $M_{2}$ a coordinate system centered in $M_{1}$ or $M_{2}$  appropiate for this study are Fermi coordinates (also used in the previous lemma), which are the generalization of normal coordinates that arises when the center of normal neighborhood is replaced by a submanifold.\\
For instance in \cite{Ge} the authors used these coordinates to compute the power series expansion formula for the shape operator of ${\bf d}^{-1} (t)$ with respect to the distance to $M_{1}$. \\
Then  explicitly  Corollary 2.2 in \cite{Ge} says that, $$h(t)= \dfrac{codim(M_{1})-1}{t}+t(trace(A)+trace(B))+o(t^2),$$
where $A,B$ are matrices independent of $t$ (if we consider the expansion close to $M_2$ instead of
$M_1$ we obtain a similar formula, only changing the sign). In particular we have the following asymptotic behaviour
of $h$ close to the focal varieties:

\begin{lemma}\label{mean-curvature}

$$\lim_{t\rightarrow 0}  t   \ h(t) = n-d_1 -1 \ \ \ ,  \ \ \ \ \lim_{t\rightarrow t^*} (t-t^* )   \ h(t) = n-d_2 -1. $$

\end{lemma}

Summarizing we see that if $x$ is not in the focal submanifolds then 

\begin{equation}
\Delta u (x) = \phi'' ({\bf d}(x)) + h ({\bf d} (x) \phi ' ({\bf d}(x))
\end{equation}

And  taking the limit we see that if if $x\in M_1 $

\begin{equation}
\Delta u (x) =(n-d_1 )  \phi'' (0) 
\end{equation}

\noindent
and if $x\in M_2 $

\begin{equation}
\Delta u (x) =(n-d_2 )  \phi'' (t^* ) 
\end{equation}

We have proved 

\begin{lemma}Let $u \in C^{2,\alpha}_f (M)$, $ u(x)=\phi({\bf d}(x))$, with $\phi \in \mathcal{B}$. Then $u$ is a solution of equation
(\ref{equation}) if and only if the function $\phi$ satisfies 

\begin{equation}\label{ODE}
-(\phi ''+ h \phi ') +\lambda\phi=\lambda \phi^{q }.
\end{equation}

\noindent 
on $[0,t^* ]$. For $t\in (0,t^* )$ $h(t)$ is the mean curvature of  ${\bf d}^{-1} (t)$.

\end{lemma}

\section{Bifurcation points}

In this section we will use local bifurcation theory (as can be found for instance  in \cite{Ambrosetti, Nirenberg}) to prove Theorem \ref{LocalBifurcationTheorem}.
We fix a proper isoparametric function $f$ on the closed Riemannian manifold $(M,g)$. As in the previous section
we denote by $t^*$ the distance between the two focal submanifolds. 
It follows from the previous section that positive $f$-invariant solutions of \ref{equation} are given by
positive solutions of the problem
\begin{align*}
\phi''+h\phi'&+\lambda(\phi^{q}-\phi)=0,\\
\phi'(0)&=0=\phi'(t^*) .
\end{align*}

For all positive constant $\lambda$ the function $\phi\equiv 1$ is a {\it trivial} solution of the equation. We will prove
Theorem \ref{LocalBifurcationTheorem} by studying bifurcation from this path of trivial solutions $\lambda \mapsto (1, \lambda )$.

\begin{proof} (Theorem  \ref{LocalBifurcationTheorem}) We define $S:\mathcal{B} \times\mathbb{R}^+\rightarrow C^{0,\alpha}([0,t^*])$ by
$$S(\phi,\lambda)=\phi'' + h\phi'+\lambda(\phi^{q}-\phi).$$

Then we are trying to solve the operator equation $S(u,\lambda )=0$.

It is easy to compute that 
$$D_{\phi}S_{(\phi,\lambda)}(u)=u'' + hu'+\lambda(q\phi^{q-1}u-u),$$
and in particular $$D_{\phi}S_{(1,\lambda)}(u)=u'' + hu'+\lambda(q-1)u.$$\medskip \\
Bifurcating branches will appear at the values of $\lambda$ for which  the kernel of linear operator $D_{\phi}S_{(1,\lambda)}$ is
nontrivial. 
Note that since  the asymptotic behavior of $h$ in $0$ is $\frac{n-d_1 -1}{t}$ the following initial value problem  has a unique solution 
\begin{align*}
u''+hu'&+\lambda(q-1)u=0,\\
u(0)&=1,\\
u'(0)&=0.
\end{align*}
The equation is of course the eigenvalue equation for the Laplacian restricted to $f$-invariant functions. 
One can see (for instance in \cite[Proposition 3.2]{Henry})
the existence of infinitely many eigenvalues of the positive  Laplacian operator in the set of $f$-invariant functions. 
We call them $\mu_k$, $k\geq 1$.
This means that $\lambda = \lambda_k = \frac{\mu_k}{q-1}$  the solutions $u_k$  of the above initial value problem 
(with $\lambda = \lambda_k$) satisfy that $u'_k(t^*)=0$. Therefore  $ker\; D_{\phi}S_{1,\lambda_k(q-1)}\neq 0$, and it has dimension one.

We can normalize  $u_k$ so that $\int_M  u_k^2=1$.  Since the operator $L=D_{\phi}S_{(1,\lambda_k(q-1))}$  is self-adjoint we have \medskip\\
\centerline{$
Range(L)=\{u\in C^{0,\alpha}(0,t^*)/ \int_M  uu_k=0\}.
$}\medskip \\
We have  that $D^2_{\phi,\lambda}S_{(1,\lambda_k(q-1))}[u_k]=(q-1)u_k\notin Range(L)$.
\medskip \\
Therefore from the well known theory of local bifurcation for simple eigenvalues (see for instance \cite[Theorem 2.8]{Ambrosetti}, 
\cite[Theorem 3.2.2]{Nirenberg} or the original article by M. G. Crandall and P. H. Rabinowitz \cite{Crandall}) we can see  that for all $k \geq 1$ $(1,\lambda_k )$ is a bifurcation point 
and moreover all nontrivial solutions in a neighborhood of $(1,\lambda_k )$ are given a branch 
$t \mapsto (u(t), \lambda (t) )$, $t\in (-\varepsilon , \varepsilon )$, such that  $\lambda (0) = \lambda_k $, $u(0) \equiv 1$,  and
$u(t) \neq 1$ if $t\neq 0$. This proves Theorem \ref{LocalBifurcationTheorem}.

\end{proof}

Later we will need the following result about the  number of zeros $n_k$ of the functions  $u_{k}$:

\begin{lemma} The sequence $n_k$ is strictly increasing.
\end{lemma}
\begin{proof}
Let $u_k, u_{k+1}$ be solutions (resp.) of
\begin{align*}
u''_{k} + hu'_{k}+\lambda_{k}(q-1)u_{k}&=0,\\
u''_{k+1} + hu'_{k+1}+\lambda_{k+1}(q-1)u_{k+1}&=0.
\end{align*}
Recall that we have set $u_k (0)=1=u_{k+1} (0)$. In particular these solutions are
non-trivial and therefore if for any $t \in (0,t^* )$,  $u_k (t) =0$ then $u_k' (t) \neq 0$. 

Let $0<t_1<t_2<\cdots<t_{n_k} <t^*$, be the points such that  $u_k (t_i ) =0$, $i = 1, \dots ,n_k$.

Since $$\dfrac{u'_{k+1}(0)}{u_{k+1}(0)}=0=\dfrac{u'_k(0)}{u_k(0)},$$

\noindent
and $\lambda_{k+1} > \lambda_k$, 
if $u_{k+1}$ does not have zeros in $(0,t_1)$ then by Sturm's  comparison  theorem (and since the
functions $u_k$ and $u_{k+1}$ are linearly independent in any open interval)
$$\dfrac{u'_{k+1}}{u_{k+1}}  < \dfrac{u'_k}{u_k}\quad\text{in}\;(0,t_1).$$

Therefore there must be at least one value $s \in (0,t_1)$ such that $u_{k+1} (s)=0$. We can
also apply the same argument to show that $u_{k+1}$ must have another zero in $(t_{n_k} , t^* )$. 
Also  by standard Sturm's comparison we can see that the function $u_{k+1}$ has at least one zero in the interval
$(t_i , t_{i+1} )$, for each $i=1, \dots ,n_{k} -1$. Therefore $u_{k+1}$ has at least $n_k +1$ zeores in $(0,t^* )$, proving
the lemma.
\end{proof}

\section{Auxiliary results}

We assume that we have a proper isoparametric function $f$ on a closed Riemannian manifold $(M,g)$. The
dimension  of the focal submanifolds are $d_1 , d_2 \leq n-2$: we call $d=\min \{ d_1 , d_2 \} \leq n-2$. And we let
$p_f = \frac{n-d +2}{n-d-2}$, $p_f = \infty$ in case $d=n-2$.
We consider equation (\ref{ODE}) with $q<p_f$. The main goal of this section is to prove the next
proposition which we will need in the following sections.

\begin{proposition}\label{compact} Let $q \in (1, p_f )$ and fix positive numbers $\varepsilon <  \lambda^*$.
The set $C= \{ \phi \in \mathcal{B} : \phi$ is positive and solves equation (\ref{ODE}) with $\lambda \in [\varepsilon ,\lambda^* ] \}$
is compact in $\mathcal{B}$.

\end{proposition} 

We first consider the case when $\lambda$ is fixed:

\begin{lemma} 
Consider equation (\ref{ODE}) with $q\in (1,p_f )$ and $\lambda = \lambda_0 >0$ fixed. If $\phi_{\alpha}$ is the solution of the
initial value problem with $\phi '(0)=0$,
$\phi (0) =\alpha$,  then there exists $A=A(\lambda_0 ) >0$ such that if $\alpha \geq A$ then there exists
$t\in (0,t^* )$ such that $\phi_{\alpha} (t) =0$. Similarly, if  $\varphi_{\alpha}$ is the solution of 
(\ref{ODE}) with $\varphi ' (t^* )=0$, $\varphi  (t^* )= \alpha$, then there exists $B>0$ such that if $\alpha \geq B$ then there exists
$t\in (0,t^* )$ such that $\varphi_{\alpha} (t) =0$.
\end{lemma}

\begin{proof} We consider the first statement, the proof of the second statement is similar.
Equation (\ref{ODE}) can be written as 

\begin{equation}\label{ODE2}
\phi '' (r)+ \frac{H(r)}{r} \phi ' (r)+ \lambda \phi^{q}(r) - \lambda\phi (r)= 0.
\end{equation}

\noindent
where $H(r)= rh(r)$ and $H(0)=n-d_1 -1$ by Lemma (\ref{mean-curvature}). 
If $d_1 = n-2$ then $H(0) =1$. If $d_1 <n-2$ then 

$$ \frac{H(0) +1}{2} =\frac{n-d_1}{2} \leq \frac{n- {\bf d}}{2} = \frac{p_f +1}{p_f -1}  < \frac{q+1}{q-1}.$$

Then we can apply \cite[Theorem 3.1]{JC} which says that under the previous  conditions on $H(0)$ and $q$
there exists $A>0$ such that if $\alpha \geq A$ then the solution $\phi_{\alpha}$ has a zero. 

\end{proof}

\begin{corollary} Fix $\lambda =\lambda_0 >0$ in equation (\ref{ODE}). There exits $A>0$ such that if $\phi  \in \mathcal{B}$ is a positive solution of  equation (\ref{ODE}) then $\phi \leq A$.
\end{corollary}

\begin{proof} It follows from the previous lemma that
 we can find $C>0$ such that  any positive solution in $\mathcal{B}$ is bounded by $C$ in
$0$ and in $t^*$. But then for any $\epsilon >0$ there exists a constant $A>0$ such that the solution
must be bounded by $A$ in 
$[0,  t^* - \epsilon ]$ and  in $[\epsilon , t^* ]$. 
\end{proof}

\begin{corollary} Fix $\lambda =\lambda_0 >0$ in equation (\ref{ODE}). The set of  $\phi  \in \mathcal{B}$ such that $\phi$ is a positive solution of  equation (\ref{ODE}) is compact.
\end{corollary}

\begin{proof}
Let $\phi_i  \in \mathcal{B}$ be sequence of positive solutions of  equation (\ref{ODE}). Then it follows  from
the lemma that we can take a subsequence so that the sequences $\phi_i (0)$, $\phi_i (t^* )$ are convergent. 
Let $\lim_{i \rightarrow \infty} \phi_i (0) = {\bf a}$,  $\lim_{i \rightarrow \infty} \phi_i (t^* ) = {\bf b}$.
But then we consider the solutions 
$\phi^1$, $\phi^2 $ of 
equation  (\ref{ODE}) with  $\phi^1 (0) ={\bf a}$, ${\phi^1}' (0) =0$, $\phi^2 (t^* ) ={\bf b}$, ${\phi^2}'  (t^* ) =0$.  Then
for any $\varepsilon >0$ small, $\phi_i$ converges in $[0, t^* -\varepsilon ]$ to $\phi^1$ and $\phi_i$ converges in $[\varepsilon , t^* ]$ to $\phi^2$. Then $\phi^1 =\phi^2$ and give a function in $ \mathcal{B}$ 
which is  a positive solution of  equation (\ref{ODE}). And (for the subsequence) 
$\lim_{i\rightarrow \infty} \phi_i = \phi^1 =\phi^2$. 
\end{proof}

\begin{remark} If $\phi  \in \mathcal{B}$ is a non-trivial positive solution of  equation (\ref{ODE}), then
$\# \{ t : \phi (t) =1 \} < \infty$, and there is an open neighborhood $U$ of $\phi \in  \mathcal{B}$ such that for
any $\varphi \in U$, $\# \{ t : \varphi (t) =1 \}  = \# \{ t : \phi (t) =1 \}$. Also if $\phi  \in \mathcal{B}$ is a non-trivial positive solution of  equation (\ref{ODE}) close to the trivial solution then $\# \{ t : \varphi (t) =1 \}$ is equal to the number
of zeroes of the linearized equation at the trivial solution (which is finite).
\end{remark}

\begin{corollary} Fix $\lambda =\lambda_0 >0$ in equation (\ref{ODE}). There exits $k_0 >0$ such that if $\phi \in \mathcal{B}$ is a positive solution of  equation (\ref{ODE}) then $\# \{ t : \phi (t) =1 \} \leq k_0 $.
\end{corollary}

\begin{proof}
Any sequence of positive solutions  $\phi_i \in \mathcal{B}$ of  equation (\ref{ODE}) must have a convergent 
subsequence. Then by the remark  $\# \{ t : \phi_i (t) =1 \}$ is bounded (independently of $i$). 
\end{proof}

Now as in the proposition we will fix positive numbers $\varepsilon  <\lambda^*$ and consider the
equation (\ref{ODE}) with $\lambda \in [\varepsilon , \lambda^* ]$.

\begin{lemma}\label{odecompa} For any $0 < \varepsilon  < \lambda^* $
there exists $A>0$ such that if $\lambda \in [\varepsilon ,\lambda^* ]$ and 
$\phi_{\alpha}$ is the solution of equation (\ref{ODE}) with $\phi_{\alpha}' (0) =0$, 
$\phi_{\alpha}  (0) = \alpha >A$ then $\phi_{\alpha}$ 
has a zero in $(0,t^* )$. Similarly, there exists $B>0$ such that if 
$\varphi_{\alpha}$ is the solution of equation (\ref{ODE}) with $\varphi_{\alpha}' (t^* ) =0$, 
$\varphi_{\alpha}  (t^*) = \alpha >B$, then $\varphi_{\alpha}$ 
has a zero in $(0,t^* )$.
\end{lemma}

\begin{proof} We will prove the first statement, the proof of the second statement is similar. For each $\lambda \in [\varepsilon,\lambda^* ]$ let $A_{\lambda} = \inf \{ A : \phi_{\alpha}$ has a zero in
$(0,t^* ) $ for avery  $\alpha > A \}$.
It follows from the previous lemma that $A_{\lambda} < \infty$ for any $\lambda \in [\varepsilon,\lambda^* ]$. Assume that there
exists a sequence $\lambda_i \in [\varepsilon , \lambda^* ]$ such that $A_{\lambda_i} \rightarrow \infty$. We can
asume that
$\lambda_i \rightarrow \lambda_0 \in [\varepsilon , \lambda^* ]$.
Then we
have a solution $\phi_i$ of equation (\ref{ODE}) with
$\lambda = \lambda_i$, such that  $\phi_i (0) = \alpha_i^{\frac{2}{q-1}} \rightarrow \infty$, and $\phi_i$ is positive in
$[0,t^* )$. Then we argue as in \cite[Theorem 3.1]{JC}
(see also \cite[Proposition 3.8]{Haraux}): We let 

$$w_i (t) = \alpha_i^{\frac{2}{1-q}}  \phi_i \left(  \frac{t}{\alpha_i\sqrt{\lambda_i}    } \right) .$$
Then $w_i$ solves 
$$w_i^{\prime\prime}(t)+\dfrac{H\left(\frac{t}{\alpha_i\sqrt{\lambda_i}}\right)}{t}w^{\prime}_i(t) + w_i^{q}(t)-\dfrac{w_i(t)}{\alpha_i^2}=0,$$

\noindent
where $H(r)= rh(r)$, $w_i(0)=1$ and $w^{\prime}_i(0)=0$. Note that $w_i$ is defined in
$[0, \alpha_i \sqrt{\lambda_i } t^* )$, and $\lim_{i \rightarrow \infty} \alpha_i \sqrt{\lambda_i } t^* =\infty$.

Then one can see that for any fixed $K>0$ $w_i$ converges uniformly on $[0,K]$ to the solution $w$
of 

$$w'' + \frac{H(0)}{t} w' + w^q =0,$$

\noindent
with $w(0)=1$, $w'(0)=0$. The proof is the same as in  \cite[Lemma 3.2]{JC}), where the proof is detailed in
the case $\lambda_i = \lambda >0$, instead of $\lambda_i \rightarrow \lambda$ as in our case
(but this does not affect the proof of the statement given in \cite{JC}).

It is proved in \cite[Proposition 3.9]{Haraux} that by picking $K$ large we can assume that $w$ has any number of zeroes in
$[0,K]$ and by the uniform convergence it follows that $w_i$ must have a zero in $[0,K]$ and
therefore $\phi_i$ has in zero in $(0, \frac{K}{\alpha_i \sqrt{\lambda_i }} )$. This is a contradiction, and therefore
$A=\sup_{\lambda \in [\varepsilon, \lambda^* ]} A_{\lambda} < \infty$, proving the lemma.

\end{proof}

We can now prove Proposition 4.1:

\begin{proof} Let $\phi_j \in \mathcal{B} $ be a sequence of solutions of equation(\ref{ODE})
with $\lambda = \lambda_j \in [\varepsilon ,\lambda^* ]$. $\phi_j$ is determined by $\alpha_j =\phi_j (0)$
and by $\beta_j = \phi_j (t^* )$. 
From the previous lemma we know that we can take a 
 subsequence and assume that $(\lambda_j , \alpha_j , \beta_j ) \rightarrow (\lambda_0 ,  \alpha_0 ,\beta_0 )$,
where $\alpha_0$, $\beta_0 >0$. 
Let $\phi^1$ be the solution of equation (\ref{ODE}) with $\lambda =\lambda_0$ such that
$\phi^1 (0) = \alpha_0$ and ${\phi^1}' (0)=0$. 
Let $\phi^2$ be the solution of equation (\ref{ODE}) with $\lambda =\lambda_0$ such that
$\phi^2 (t^* ) = \beta_0$ and ${\phi^2}' (t^* )=0$. 
Then for any $\delta >0$ $\phi_j$ converges on $[0, t^* -\delta ]$ to $\phi^1$ and
on $[\varepsilon , t^* ]$ to $\phi^2$. It follows that on $(0, t^* )$ $\phi^1 = \phi^2$ and therefore 
they define a function  $\phi \in \mathcal{B} $ which is positive, solves
equation (\ref{ODE}), and verfies $\phi (0)= \alpha_0$, $\phi (t^* ) =\beta_0 $. And $\phi_i \rightarrow \phi$.

\end{proof}

\section{\texorpdfstring{$f$}{}-invariant solutions of equation (\ref{equation}) for $\lambda$ close to zero}

In this section we will prove Theorem 1.2: i.e. we will show that all non-negative $f-$invariant solutions of \ref{equation} are constant for $\lambda$ close to zero. In order to get this result we will give an apriori estimate for $f-$invariant solutions of the
equation. First consider the equivalent equation: 

\begin{equation}\label{eqconstant}
-\Delta_g w=w^{q}-\lambda w.
\end{equation}

Note that $u$ is a solution of equation \ref{eqconstant} if and only if $\lambda^{\frac{-1}{q-1}} u$ is a solution
of equation \ref{equation}. Also, as in section 2, $u \in C^{2,\alpha}_f (M)$, $ u(x)=\phi({\bf d}(x))$, with $\phi \in \mathcal{B}$ is a solution of equation
(\ref{eqconstant}) if and only if the function $\phi$ satisfies 

\begin{equation}\label{ODE2}
-(\phi ''+  h \phi ') +\lambda\phi= \phi^{q }.
\end{equation}

\noindent 
on $[0,t^* ]$. 

\vspace{.5cm}

Similar problems have been considered before, for instance in \cite{Licois}. We denote $\lambda_1$  the first non-zero eigenvalue of $-\Delta_g$.  We will make use of
the following result from \cite[Theorem 2.2]{Licois}

\begin{theorem}[\cite{Licois}]
Assume $0 < \lambda$ and $q>1$. If $w$ is solution of (\ref{eqconstant}) which satifies $$q \Vert w\Vert^{\frac{1}{q-1}}_{L^{\infty}}\leq \lambda+\lambda_1  , $$
then $w=\lambda^{\frac{1}{q-1}}.$

Similarly, if  $u$ is solution of (\ref{equation}) which satifies $$q \Vert 
u\Vert^{\frac{1}{q-1}}_{L^{\infty}}  \leq  \frac{  \lambda+\lambda_1 }{\lambda^{\frac{1}{(q-1)^2 }}}, $$
then $u= 1.$

\end{theorem}

We first find an appropriate bound for positive $f$-invariant solutions:

\begin{lemma}Let $q\in (1, p_f )$. 
There exist constants $\varepsilon, c>0$ such that for $\lambda\in(0,\epsilon]$ any positive $f-$invariant solution $w$ of (\ref{eqconstant}) satisfies $$w\leq c\lambda^{\frac{1}{q-1}}.$$

If $u$ is a positive  solution of equation (\ref{equation}) then $u \leq c$.
\end{lemma}

\begin{proof}
We follow  a similar treatment of \cite[Theorem 2.3]{Licois}. Suppose that the lemma is not true. Then we 
have a sequence of positive numbers  $\lambda_m  \rightarrow 0$, a sequence of positive numbers  $c_m \rightarrow + \infty$, and a sequence $p_m \in M$, such that there
exists a positive solution $w_m$ of (\ref{eqconstant}) which satisfies that \begin{equation}\label{max}
\max_M\;w_m=w_m(p_m)=c_m\lambda_m^{\frac{1}{q-1}}.
\end{equation}

By taking a subsequence we can assume  that $w_m(p_m)\rightarrow {\bf a} \in [0, \infty ]$. 

Note
that if ${\bf a}=0$ then the solutions $w_m$ satisfies the conditions on Theorem 5.1 and therefore
we would have $w_m=\lambda_m^{\frac{1}{q-1}}$ for $m$ large. This would say that $c_m =1$,
which is a contradiction (we were assuming that $c_m\rightarrow\infty$).

Assume now that ${\bf a} \in (0, \infty )$. Let  $v_m=\frac{w_m}{w_m(p_m)}$. Note that $v_m$ solves the
equation 

$$\Delta_g v_m-\lambda_mv_m+w_m(p_m)^{q-1}v_m^{q}=0,$$ 
with $\Vert v_m\Vert_{L^\infty}=1$. Recall that $\lambda_m\rightarrow 0$ and $w_m(p_m) \rightarrow
{\bf a}$. 
Therefore, from the theory of elliptic operators we get that the sequence 
$v_m$ converges  to a function $v$ which is a non-negative solution of $$\Delta_g v+{\bf a}^{q-1}v^{q}=0$$
on $M$.
Since $M$ is closed, $v$ is equal to zero, but this is a contradiction since $\Vert v_m\Vert_{L^\infty}=1$.

Therefore we can assume that $w_m(p_m)\rightarrow  \infty $. In this case we will need to use our hypothesis
that the functions $w_m$ are $f$-invariant and $q<p_f$. Then $w_m$ is determined by a function $\phi_m \in
\mathcal{B}$ which solves equation (\ref{ODE2}). The function $\phi_m$ is determined by
$\alpha_m = \phi_m (0)$ and by $\beta_m = \phi_m (t^* )$. If the sequences $\alpha_m$ and $\beta_m$ 
are bounded then $\phi_m$ would be uniformly bounded, which is not the case. We can therefore  assume for
instance that $\alpha_m \rightarrow \infty$. 
 
We call $H(t) = th(t)$. Then 

$$\phi_m^{\prime\prime}(t)+\dfrac{H(t)}{t}\phi_m^{\prime}(t) + \phi_m^{q}(t)-\lambda_m \phi_m(t) =0.$$

We let $\delta_m = \alpha_m^{\frac{2}{q-1}}$ and 

 $$\varphi_m (t):=\delta_m^{\frac{2}{1-q}}\phi_m\left(\frac{t}{\delta_m}\right).$$

Then $\varphi_m$ solves

$$\varphi_m^{\prime\prime}(t)+\dfrac{H(\frac{t}{\delta_m})}{t}\varphi_m^{\prime}(t) + \varphi_m^{q}(t)-\dfrac{\lambda_m}{\delta_m^2}\varphi_m(t) =0,$$ 
and satisfies  $\varphi_m(0)=1$ and $\varphi_m^{\prime}(0)=0$.\\

Since $\delta_m \rightarrow \infty$ and $\lambda_m \rightarrow 0$ we can argue as in the proof of Lemma 4.7,
or \cite[Lemma 3.2]{JC}, to prove that the sequence  $\varphi_m$ converges uniformly on any
compact interval $[0,K]$ to the solution $w$ of

$$w'' + \frac{H(0)}{t} w' + w^q =0,$$

\noindent
with $w(0)=1$, $w'(0)=0$. We recall that it is proved in \cite[Proposition 3.9]{Haraux} that by picking $K$ large we can assume that $w$ has any number of zeroes in
$[0,K]$ and by the uniform convergence it follows that for $m$ large enough $\varphi_m$ must have a zero in $[0,K]$. This
implies that $\phi_m$ has a zero in $(0, \frac{K}{\delta_m} )$. This contradicts our assumption that
the solution $w_m(x)=\phi_m({\bf d}(x))$ was positive, finishing the proof of the lemma.

\end{proof}

We are now ready to prove Theorem 1.2:

\begin{proof} By the previous lemma there exists $\varepsilon >0$ such that if $\lambda \in 
(0,\varepsilon ]$ then any positive $f-$invariant solution $w$ of \ref{eqconstant} satisfies $$w\leq c\lambda^{\frac{1}{q-1}},$$

\noindent
for some positive constant $c$ independent of $\lambda$. There exists $\lambda_0 \in (0, \varepsilon )$
such that if $\lambda \in (0,\lambda_0 ]$ then $ q (c\lambda^{\frac{1}{q-1}} )^{\frac{1}{q-1}}  \leq \lambda+\lambda_1   $. Then it follows from Theorem 5.1 that if $\lambda \in (0,\lambda_0)$ and $w$ is a solution of
equation (\ref{equation}) then it must be constant.

\end{proof}

\section{Global bifurcation}

\noindent
In this section we will prove prove Theorem 1.3. We have seen in Section 2 that $f$-invariant solutions of
equation (\ref{equation}) are given by positive solutions of equation (\ref{ODE}).

Let $D =\{ (\phi , \lambda) \in (\mathcal{B} - \{ 1 \}  \times (0, \infty ) : \phi \in \mathcal{B} \;\text{is a positive 
nontrivial solution of}\; (\ref{ODE}) \}$.
Let $\overline{D}$ be the closure of $D$  in $\mathcal{B}$ and $D_k$ the connected component of $\overline{D}$ containing the bifurcation point  $(1,\lambda_k)$ (as in Section 3). 

It follows from Theorem 1.2 that there exists
$\varepsilon >0$ such that for any $k \geq 1$,  $D_k$ is 
contained in $\{ \mathcal{B}  \times [\varepsilon, \infty ) \} $.\\
Now we shall see that each $D_k$ is not compact, using the global bifurcation theorem of P. Rabinowitz (see \cite[Theorem 3.4.1]{Nirenberg}, \cite[Theorem 4.8]{Ambrosetti} or \cite{Rabinowitz}) .

It follows from  Rabinowitz's theorem  that either

a) $D_k$ is not compact in $O=\{(\phi,\lambda)\in \mathcal{B} \times \mathbb{R}^+/\;\phi>0\}$ 

 or

b) $D_k$ contains a point $(1,\lambda_j )$ for $j\neq k$ .\medskip\\

For each $k\geq 1$, we let  $$Z_k:=\{ \phi \in \mathcal{B} \;/ \phi-1 \;\text{has exactly }\; n_k\; \text{simple zeros  in}\;(0,t^*) \; \},$$

\noindent
where we recall from Section 3 that $n_k$ is the number of zeroes of the solution of the linearized
equation at $(1,\lambda_k )$.  
 
Each $Z_k$ is an open set in $\mathcal{B}$. Note also that if $\phi \in \mathcal{B}$ is a nontrivial solution
of (\ref{ODE}) then the zeros of $\phi -1$ are simple (since it solves a second order ordinary differential equation
for which the constant function 1 is a solution).

Recall from Section 3 (the proof of Theorem 1.1) that
the points in $D$ near to $(1,\lambda_k)$ can be parametrized by a curve  $s \mapsto (v_k (s),\mu_k (s))$, $|s|<\varepsilon_k$, where $\mu_k(0)=\lambda_k$. The map $v_k$ is of the form  $v_k(s)=1+ s u_k+sQ(s)$ and $Q(0)=0$, 
where, as in Section 3,   $u_k\in Z_k $ is an eigenfunction of $-\Delta$  associated to $\lambda_k (q-1)$ 
and therefore it has exactly $n_k$ zeroes (see again \cite[Theorem 3.2.2]{Nirenberg}, for instance).

Therefore, $v_k (s)\in Z_k$ for $s$ sufficiently small, $s \neq 0$. Then it follows that $D_k - \{ (1, \lambda_k )  \}\subseteq Z_k$.
And in particular it follows that $D_i \cap D_j = \emptyset$ if $i \neq j$. This says that alternative (b) in the 
global Theorem of Rabinowitz does not happen
and therefore $D_k$ is not compact, for any $k\geq 1$. 

If there exists a constant  $\lambda_0>\lambda_k$ such that for any $(\phi , \lambda ) \in D_k$ we have $\lambda \leq \lambda_0$. Then $D_k$ 
would be a closed set of $\{ (\phi , \lambda ) \in \mathcal{B} \times [\varepsilon , \lambda_0 ]:  \phi$ is a positive solution of (\ref{ODE}) $\}$. Then it 
follows from Proposition 4.1 that $D_k$ is compact. Therefore such  $\lambda_0$ does not exist and since
$D_k$ is connected it follows that for any $\lambda \geq  \lambda_k $ there exists $(\phi , \lambda ) \in D_k $. 
Then for $\lambda \in [\lambda_i  ,\lambda_{i+1}  )$ and for each $k\leq i$ there exists $(\phi , \lambda ) \in D_k$ and this
proves Theorem 1.3.

\section{Eigenvalues of restricted Laplacians}

Assume we have an isoparametric function $f$ on a Riemannian manifold $(M,g)$.
In order to apply Theorem \ref{GlobalBifurcationTheorem} we need to understand  the eigenvalues of $\Delta_g$ 
restricted to $f$-invariant functions and the dimension of the focal submanifolds. While the dimension of the focal submanifolds
is usually simple to understand, to compute the eigenvalues of the restricted Laplacian might be lengthy. In this section we
will first show how to do these calculations in some particular cases.

The  situation we will consider is a Riemannian
submersion with totally geodesic fibers $\pi : (M_1 , g_1 ) \rightarrow (M_2 ,g_2 )$. In this situation the corresponding
Laplacians commute: for any function $f: M_2 \rightarrow \re$, $\Delta_{g_2}  (f) \  \circ \pi = \Delta_{g_1} (f \circ \pi )$. 
And it is easy to check that $f$ is isoparametric for $(M_1 , g_1 )$ if and only if $f\circ \pi$ is isoparametric for
$(M_2 , g_2 )$. Actually $\| \nabla f \|^2  = a \circ f$ and $\Delta f = b \circ f$ if and only if 
$\| \nabla (f  \circ \pi ) \|^2  = a \circ ( f \circ \pi )$ and $\Delta  ( f \circ \pi ) = b \circ  ( f \circ \pi )$. Then it follows easily that 
$h = \alpha \circ f$ is an eigenfunction of $\Delta_{g_2}$ with eigenvalue $\lambda$ if and only if  $h \circ \pi$ is
an eigenfunction of $\Delta_{g_1}$ with eigenvalue $\lambda$. So it is equivalent to study $f$-invariant eigenfunctions
or $(f\circ \pi )$-invariant eigenfunctions. 

The other fact we will use is that the problem is easy to solve in the case of the round sphere. In general if $f$ is an
isoparametric function then one can consider the family of isoparametric functions of the form $\alpha \circ f$, 
where $\alpha$ is a monotone function. These isoparametric functions
are in certain sense equivalent: they have the same level sets and the spaces of $f$-invariant functions and
$(\alpha \circ f)$-invariant functions are the same. In the case of the round sphere, $(\mathbb{S}^n , g_0 )$, there is a canonical way to pick a
representative of these families of {\it equivalent} isoparametric functions. Namely, in any such family H. F. 
M\"{u}nzner (\cite{m1,m2}) proved that there is a 
{\it Cartan-M\"{u}nzner polynomial}. This is a homogeneous harmonic polynomial $F$  (in $\re^{n+1}$ ) of degree $k$ which solves the 
Cartan-M\"{u}nzner equations:

$$\| \nabla F  (x) \|^2 = k^2 \| x \|^{2k-2} $$

$$\Delta F (x) = \frac{1}{2} ck^2 \| x \|^{k-2} ,$$

\noindent
for some integer $c$.
But then one can easily see by studying the resulting linear ordinary differential equation that the $F$-invariant eigenvalues are exactly $\mu_i =
\lambda_{ki}$, $i \geq 1$, where $\lambda_j = j(n+j-1)$ are the eigenvalues of $-\Delta_{(\mathbb{S}^n , g_0 )}$  (see \cite[Lemma 3.4]{Henry}). 

\vspace{1cm}

Let us now consider  the case of the complex projective
spaces with the Fubini-Study metric $({\bf CP}^n ,g_{FS} )$.  Recall that the positive eigenvalues of
$-\Delta_{g_{FS}}$ are $2i(2i+2n)$, $i\geq 1$ (see for instance  \cite{besse1978}). There is  a Riemannian submersion (the Hopf fibration)
$\mathbb{S}^{2n+1} \rightarrow ({\bf CP}^n ,g_{FS} )$, obtained by considering the canonical diagonal $\mathbb{S}^1$-action on 
$\mathbb{S}^{2n+1}$. 
It has totally geodesic fibers (which are circles, the orbits of the $\mathbb{S}^1$-action) so we can apply the previous ideas. An isoparametric
function $f$ on $({\bf CP}^n ,g_{FS} )$ lifts to an isoparametric function $\overline{f} : (\mathbb{S}^{2n+1} ,g_0 ) \rightarrow [t_0 , t_1 ]$.
And we can look for the corresponding Cartan-M\"{u}nzner polynomial.   

We will consider the three simplest examples of isoparametric functions on  $({\bf CP}^n ,g_{FS} )$.  These are given by
cohomogeneity one actions.

1) Let us consider first  the action of $U(n) \subset SO(2n)$. This action lifts to a cohomogeneity one
action on $\mathbb{S}^{2n+1}$ which commutes with  the diagonal $\mathbb{S}^1$-action (the action on $S^{2n+1} \subset {\bf R}^{2n+2}$
is given by $A. (x_1 , x_2 , y_1, ....y_{2n} ) = (x, Ay)$. 
We consider on $\re^{2n+2}$ the homogeneous harmonic polynomial $F(x,y) = \| x \|^2 -\| y \|^2$. It is invariant by the action 
of $\mathbb{S}^1 \times O(2n)$
and therefore projects to an isoparametric function $f$ on $({\bf CP}^n ,g_{FS} )$ invariant by the 
$U(n)$-action. $F$ is a Cartan-M\"{u}nzner polynomial of degree 2. Then it follows that the $f$-invariant eigenvalues of
$-\Delta_{g_{FS}}$ are $\lambda_{2i}=2i(2i+2n)$. Note that these are actually the eigenvalues of the full Laplacian 
$-\Delta_{g_{FS}}$. Also note that the action of $U(n)$ on $({\bf CP}^n ,g_{FS} )$ has a fixed point, so $p_f = p_{2n}$.

2) Let us now consider the action of $U(k) \times U(l) \subset U(n+1)$ on $({\bf CP}^n ,g_{FS} )$, where we ask $n\geq 3$,
$k+l =n+1$ and $k\geq l \geq 2$. Similarly to the previous case we can easily lift the action to  
$\mathbb{S}^{2n+1} \subset {\bf R}^{2n+2}$, commuting with the diagonal $\mathbb{S}^1$-action. The action looks like $(A,B). (x,y)=
(Ax,By)$. Again $F(x,y) = \| x \|^2 - \|y \|^2$ is an invariant Cartan-M\"{u}nzner polynomial of degree 2 which projects
to an isoparametric function $f$ on $({\bf CP}^n ,g_{FS} )$.  It follows that the $f$-invariant eigenvalues of
$-\Delta_{g_{FS}}$ are $\lambda_{2i}=2i(2i+2n)$. But now note that the critical orbits are ${\bf CP}^{k-1}$ and
${\bf CP}^{l-1}$. Therefore $p_f = \frac{2n-2l +4}{2n-2l}>p_{2n}$.

3)  There is a cohomogeneity one isometric action of $SO(n+1) \subset U(n+1)$ in $({\bf CP}^n ,g_{FS} )$ given by considering the natural action on $\mathbb{C}^{n+1}$(in the introduction we considered the case $n=2$). This action can obviously be lifted to $\mathbb{S}^{2n+1}\subseteq \mathbb{R}^{2n+2}$. The corresponding isoparametric polynomial on the sphere $\mathbb{S}^{2n+1}$ is given by $F(x,y)=( \Vert x \Vert^2 - \Vert y \Vert^2)^2+4\langle x,y\rangle^2$, which is clearly invariant under the action of $SO(n+1)$. It follows that the $f$-invariant eigenvalues of $-\Delta_{g_{FS}}$ are $\lambda_{4i}=4i(4i+2n)$. The singular orbits for this action are ${\bf RP}^n$ and the Grassmanian of oriented two-planes $\widetilde{\textnormal{Gr}}(2,\mathbb{R}^{n+1})$, which have dimensions $n$ and $2n-2$, respectively. Hence, $d=n$ and $p_f=\frac{n+2}{n-2}>p_{2n}$. Then applying Theorem \ref{GlobalBifurcationTheorem} we obtain:

\begin{corollary}\label{7}
Let $q \in (1, \frac{n+2}{n-2} )$. Equation (\ref{equation}) on $({\bf CP}^n ,g_{FS} )$ has at least $k$ positive 
different $SO(n+1)$-invariant solutions if $\lambda \in \left( \frac{4k(4k+2n)}{q-1} , \frac{4(k+1)(4(k+1) +2n)}{q-1} \right]$.
\end{corollary}

\vspace{1cm}

Next we turn to the quaternionic projective space. The quaternionic projective space, ${\bf HP}^n$, can be endowed with a metric analogous to the Fubini-Study metric for the complex projective space. We denote this metric by $h_{FS}$. It  is an Einstein metric with positive scalar curvature.  It is also well-known that the positive eigenvalues of
$-\Delta_{h_{FS}}$ are $2i(2i+4n+2)$, $i \geq 1$ (see for instance \cite{besse1978}). There is a  Hopf fibration
$$ (\mathbb{S}^{3} g_0 ) \rightarrow (\mathbb{S}^{4n+3} ,g_0 ) \rightarrow ({\bf HP}^n ,h_{FS} ),$$
which is  a Riemannian submersion with totally geodesic fibers, so we can apply again the same ideas. 
We will consider first the isoparametric functions on $({\bf HP}^n , h_{FS})$ which come from cohomogeneity one isometric
actions. We have: 

1) the action of $Sp(n)$. 

2) the action of $Sp(k) \times Sp(l)$, where $k+l = n+1$, $k, l \geq 2$

3) the action of $U(n+1) \subset Sp(n)$. 

But there are also isoparametric functions which are not homogeneous, i. e. they are not invariant by any cohomogeneity one
action on $({\bf HP}^n , h_{FS})$. We will discuss the simplest known example:

4) The first  inhomogeneous example was found by H. Ozeki and M. Takeuchi in  \cite{OT1975}.  In what follows, $\langle u,v \rangle_{\mathbb{H}}= \sum u_i \bar{v}_i$ denotes the inner product of two vectors with quaternionic entries. We will also denote $u=(u_0, u_1)$, $v=(v_0,v_1)$ where $u_0,v_0 \in \mathbb{H}$ and $u_1,v_1 \in \mathbb{H}^n$. We then define $F_0: \mathbb{H}^{n+1}\times \mathbb{H}^{n+1} \rightarrow \mathbb{R}$ by
$$ F_0(u,v):= 4 \left( | \Im \langle u,v \rangle_{\mathbb{H}}|^2 \right)+ \left( \|u_1\|^2-\|v_1\|^2 + (u_0 \bar{v}_0 + v_0 \bar{u}_0) \right)^2,$$
where $\Im q$ denotes the imaginary part of the quaternion $q$. The polynomial $F(u,v):= n^4-2F_0(u,v)$ gives a non-homogeneous isoparametric function on $ (\mathbb{S}^{4n+3} ,g_0 )$ which
commutes with the diagonal action of $Sp(1) \subset O(4)$ and therefore induces an isoparametric function on $({\bf HP}^n , h_{FS})$.

\vspace{.5cm}

We now do the computations in these four cases. The first two are simple:

1) The action of $Sp(n)$ on lifts to an isometric action on $\mathbb{S}^{4n+3}$ which commutes with the diagonal
action of $Sp(1)$ ($Sp(1) \subset SO(4)$ acts diagonally on $\re^{4n+4} = ( \re^4 )^{n+1}$). The
polynomial  $F(x_1 ,...x_4 , y_1,...,y_{4n} ) = \|  x \|^2 - \| y \|^2 $ is a Cartan-M\"{u}nzner polynomial on $\mathbb{S}^{4n+3}$,
invariant by the action of $Sp(1) \times Sp(n)$. It projects to an isoparametric function $f$
on $({\bf HP}^n , h_{FS})$ invariant by the $Sp(n)$-action. The focal varieties of $f$ are a point and 
 ${\bf HP}^{n-1}$, so $p_f =p_{4n}$, and since $F$ has degree 2, the $f$-invariant eigenvalues of $-\Delta_{h_{FS}}$ are
$\lambda_{2i} = 2i(2i+4n+2)$, which are the same as the eigenvalues of the full Laplacian.

2) Again the action of $Sp(k) \times Sp(l)$ lifts to an isometric action on $\mathbb{S}^{4n+3}$ which commutes with the diagonal
action of $Sp(1)$. 
 We assume without loss of generality that $k\geq l \geq 2$. The
polynomial $F(x_1 ,...x_4 , y_1,...,y_{4n} ) = \| x \|^2 - \| y \|^2 $ is a Cartan-M\"{u}nzner polynomial on $\mathbb{S}^{4n+3}$,
invariant by the action of $Sp(1) \times Sp(k) \times Sp(l)$. It projects to an isoparametric function $f$
on $({\bf HP}^n , h_{FS})$ which is invariant by the $Sp(k) \times Sp(l)$-action. The focal varieties of
$f$ are ${\bf HP}^{k-1}$ and ${\bf HP}^{l-1}$. Therefore the  $f$-invariant eigenvalues of $-\Delta_{h_{FS}}$ are
$\lambda_{2i} =2i(2i+4n+2) $, which are the same as the eigenvalues of the full Laplacian, and $p_f = \frac{4n-4l+6}{4n-4l+2} >p_{4n}$.

\vspace{1cm}

The other examples are more complicated. The corresponding  Cartan-M\"{u}nzner polynomials will have degree 4. 
In the case (4) we have given it explicitly, and the corresponding isoparametric function on the quaternionic
space is the projection of the polynomial. In case (3) we must find a Cartan-M\"{u}nzner polynomial on
$\mathbb{S}^{4n+3}$ which projects to a $U(n+1)$-invariant isoparametric function (on $({\bf HP}^n , h_{FS})$).
 They will be
examples of what are called FKM-polynomials. We give a brief description for completeness, details can be found for 
instance in \cite{Cecil}.

In \cite{FKM1981}  a class of isoparametric functions on the sphere was described in terms of representations of Clifford algebras. More precisely, the authors produced isoparametric polynomials of degree 4 that solve
the Cartan-M\"{u}nzner equations, which are expressed in terms of Clifford systems in the following fashion. The first step in their construction is to fix a representation of the real Clifford algebra $\rho : \mathcal{C}\ell(V) \rightarrow O(l) $, where $V$ is a finite dimensional vector space spanned by an orthonormal basis $e_1, \ldots , e_{m-1}$. The image of this orthonormal basis under the representation provides a set of matrices $E_1, \ldots , E_{m-1} \in O(l)$. 

A Clifford system is a set of symmetric matrices $P_0, \ldots , P_m$ such that 
$$ P_i P_j + P_j P_i = 2 \delta_{ij} I_{2l}  \textnormal{ for every } 1 \leq i,j \leq m.$$    
Clifford systems are in one to one correspondence with representations of Clifford algebras. Indeed, given a representation of $\rho : \mathcal{C}\ell(V) \rightarrow O(l) $ as above we obtain a Clifford system by taking
$$ P_0= \left( \begin{array}{cc}  I_{l} & 0 \\ 0 & -I_l \end{array} \right),\qquad P_1= \left( \begin{array}{cc}  0 & I_l \\  I_l & 0 \end{array} \right), \qquad \textnormal{and} \qquad P_{1+k}= \left( \begin{array}{cc}  0 & E_k \\  -E_k & 0 \end{array} \right) $$
for $1 \leq k \leq m-1$. Conversely, a Clifford system determines a representation of $\mathcal{C}\ell (V)$ by looking at the $+1$ eigenspaces of the matrices $P_i$ (note that since $P_i^2=I_{2l}$ the only eigenvalues can be $\pm 1$). We skip the details of this construction since we will not be needing them. The interested reader can consult \cite{FKM1981}.

The main result in \cite{FKM1981} is given by

\begin{theorem}[Theorem 4.1 in \cite{FKM1981}]
Let $P_0 , \ldots , P_m$ be a Clifford system in $\mathbb{R}^{2l}$ and set $m_1=m$, $m_2=l-m_1 - 1$. Assume that $m_2>0$. The polynomial $F:\mathbb{R}^{2l} \rightarrow \mathbb{R}$ given by 
$$ F(x):= \| x \|^4 -2 \left( \sum_{i=0}^m \langle P_i x, x \rangle^2 \right). $$
solves the Cartan-M\"{u}nzner equations (and therefore gives an isoparametric function on $(\mathbb{S}^{2l-1} , g_0 )$.

\end{theorem}
We recall some of the most important properties of these polynomials (which we will henceforth refer to as FKM polynomials). The first one is that the polynomial $F$ above does not depend directly on the Clifford system $P_0, \ldots , P_m$ but rather on the Clifford sphere determined by it. This sphere is just the unit sphere in the vector space $E\subseteq \mathfrak{gl}(2l,\mathbb{R})$ spanned by $P_0, \dots , P_m$, where the metric being used is the standard Frobenius inner product 
$$ \langle A , B \rangle := \frac{1}{2l	} \tr (AB^{T}). $$
Hence, two Clifford systems with the same Clifford sphere determine the same FKM polynomial. Henceforth we denote the Clifford sphere corresponding to a Clifford system by $\Sigma (P_0 , \ldots , P_m)$. 

Another important property that the FKM polynomials have is that they are invariant under the action of the Clifford sphere, that is, 
$$ F(Qx)=F(x)$$
for any $Q \in \Sigma (P_0, \ldots , P_m )$. Finally, we point out that given a Clifford system, the vector fields $x \mapsto P_i P_j x$ define Killing fields on the sphere $\mathbb{S}^{2l-1}$. Furthermore, the integral curves of these vector fields lie on the level sets of the corresponding FKM polynomial. These Killing fields are easily shown to form a subalgebra of $\mathfrak{so}(2l)$ which is isomorphic to $\mathfrak{so}(m+1)$. Hence, this shows that the level sets of $F$ are invariant under the action of $\textnormal{Spin}(m+1)$. \\

We consider the Hopf fibration $S^3 \rightarrow S^{4n+3} \rightarrow \hpn $  and will be interested in determining 
the FKM polynomials which project to isoparametric functions on the quaternion projective spaces. In particular
we are interested in finding such a polynomial that projects to an isoparametric function invariant by the action
of $U(m+1)$ on $\hpn$.
We will take set $l=2n+2$ for an arbitrary $n\geq 1$. We are interested in FKM polynomials that are invariant under the action of $\textnormal{Spin}(3) = Sp(1)$, where we consider this action as right multiplication on $\mathbb{H}^{n+1}$ by unit the quaternions. We  fix $m=2$ above, and look at representations of $\mathcal{C}\ell (\mathbb{R})=\mathbb{C} \rightarrow O(2n+2)$. It is clear that such a representation must be (up to equivalence) a direct sum of the irreducible representation
$$ i \mapsto \left( \begin{array}{cc}  0 & -1 \\ 1 & 0 \end{array} \right).$$
The Clifford system associated to this irreducible representation is obtained as outlined above and is given by 
$$ A_0:= \left( \begin{array}{cccc}
1 & 0 & 0 & 0\\ 0 & 1 & 0 & 0 \\ 0 & 0 & -1 & 0 \\ 0 & 0 & 0 & -1 
\end{array} \right), \qquad A_1:= \left( \begin{array}{cccc}
0 & 0 & 1 & 0\\ 0 & 0 & 0 & 1 \\ 1 & 0 & 0 & 0 \\ 0 & 1 & 0 & 0 
\end{array} \right), \qquad A_2:= \left( \begin{array}{cccc}
0 & 0 & 0 & -1\\ 0 & 0 & 1 & 0 \\ 0 & 1 & 0 & 0 \\ -1 & 0 & 0 & 0 
\end{array} \right).$$
By adding up this irreducible Clifford system $n+1$ times we obtain the block diagonal matrices  $P_0:= \diag (A_0, \ldots, A_0 )$, $P_1:= \diag (A_1, \ldots, A_1)$, $P_2:= \diag (A_2, \ldots, A_2 )$ which define a Clifford system on $\mathfrak{gl}(4n+4,\mathbb{R})$. It is possible to write the FKM polynomial corresponding to this Clifford system in different ways using the the identifications $\mathbb{H}^{n+1} \cong \mathbb{C}^{n+1}\oplus \mathbb{C}^{n+1} \cong \mathbb{R}^{2n+2} \oplus \mathbb{R}^{2n+2}$. For the latter case (i.e the real  interpretation) it is easy to see that $F: \mathbb{R}^{2n+2}\oplus \mathbb{R}^{2n+2} \rightarrow \mathbb{R}$ can be written as 
\begin{equation}\label{FKM2}
F(x)= \| x \|^4-2 \left( \left( \| X \|^2 - \| Y \|^2  \right)^2 +4 \langle X, Y \rangle^2 +4 \langle X^\perp , Y \rangle^2 \right),
\end{equation}
where $ x=(x_0,x_1,y_0,y_1, \ldots )$ represents a vector in $\mathbb{R}^{4n+4}$, $X=(x_0, x_1 , \ldots )$, $Y=(y_0, y_1 ,\ldots)$ are vectors on $\mathbb{R}^{2n+2}$, and $X^\perp =(-x_1, x_0 , \ldots )$. The last vector $X^\perp$ can be obtained by the inclusion $X \in \mathbb{C}^{n+1}$ and multiplying by $i$. We have: 

\begin{proposition}
The FKM polynomial \eqref{FKM2} considered as a polynomial $F:\mathbb{H}^{n+1}\rightarrow \mathbb{R}$ is invariant under right multiplication by the unit quaternions $Sp(1)$. This polynomial is also invariant under the action of the  unitary group $U(n+1) \subseteq Sp(n+1)$ and both actions commute.
\end{proposition}
\begin{proof}
We begin by describing the action of $U(n+1)$. The invariance of $F$ becomes obvious if we introduce the complex coordinates $z=(x_0+x_1i, x_2+x_3i,\ldots )$ and $w=(y_0+x_1i, y_2+y_3i,\ldots )$. It is easy to see that 
$$ F(x)=F(z,w) = \| (z,w) \|^4 -2 \left(\left( \| z \|^2 - \| w \|^2 \right)^2 +4 \vert \langle z,w \rangle_{\mathbb{C}} \vert^2 \right),$$
where, $\langle u , v \rangle_{\mathbb{C}}= \sum u_i \bar{v_i}$ is the standard Hermitian inner product on $\mathbb{C}^{n+1}$. This shows that $F(Az,Aw)=F(z,w)$ for $A \in U(n+1)$.\\

To see that $Sp(1)$ acts by right multiplication consider the matrices 
\begin{align*} 
X_i:=A_1 A_2=\left( \begin{array}{cccc}
0 & -1 & 0 & 0\\ 1 & 0 & 0 & 0 \\ 0 & 0 & 0 & 1 \\ 0 & 0 & -1 & 0 
\end{array} \right)&, \qquad X_j:= A_1 A_0=\left( \begin{array}{cccc}
0 & 0 & -1 & 0\\ 0 & 0 & 0 & -1 \\ 1 & 0 & 0 & 0 \\ 0 & 1 & 0 & 0 
\end{array} \right),\\ X_k:= A_2 A_0=&\left( \begin{array}{cccc}
0 & 0 & 0 & -1\\ 0 & 0 & 1 & 0 \\ 0 & -1 & 0 & 0 \\ 1 & 0 & 0 & 0 
\end{array} \right).
\end{align*}
These matrices correspond to right multiplication by $i,j$, and $k$, respectively under the standard identification $\mathbb{H}=\mathbb{R}^4$ and they span a Lie subalgebra of $\mathfrak{so}(4)$ which is isomorphic to $\mathfrak{so}(3)$. Furthermore, using Taylor expansions it is straightforward to see that the following identities hold:
\begin{align*}
\exp (tX_i )& =\cos (t) I_{4} + \sin (t) X_i \\
\exp (tX_j )& =\cos (t) I_{4} + \sin (t) X_j \\
\exp (tX_k )& =\cos (t) I_{4} + \sin (t) X_k. \\
\end{align*}
These identities remain valid if we replace each $A_i$ above with the block diagonal matrix $P_i$. Furthermore, the matrices on the right hand sides can be factored as products of two elements on the Clifford sphere $\Sigma (P_0,P_1,P_2)$, which readily shows the $Sp(1)$ invariance. 

%The elements on the right hand sides are all in the Clifford sphere $\Sigma (A_0,A_1,A_2)$. Taking block diagonal matrices as before, it is easy to see that the right multiplication $x \cdot s$ where $x\in \mathbb{H}^{n+1}$ and $s \in Sp(1)$ is represented in terms of matrices as multiplication by an element of the form $\exp (\theta)$, where $\theta$ is in the Lie subalgebra spanned by the matrices $\{ P_i P_j \}_{0\leq i,j \leq 2}$. This shows that $F$ is invariant under right multiplication by $Sp(1)$.
%$F(e^\theta \cdot x)=F(x)$  , 
\end{proof}

We come back now to the computations.

3) As proved in the previous theorem the polynomial 

$$F(x)= \| x \|^4-2 \left( \left( \| X \|^2 - \| Y \|^2  \right)^2 +4 \langle X, Y \rangle^2 +4 \langle X^\perp , Y \rangle^2 \right),$$

\noindent
on  $\re^{4n+4}$ is invariant under the action of $Sp(1)$ and therefore projects to an isoparametric function $f$ on ${\bf HP}^n$ which is invariant under $U(n+1)$. Since $F$ has degree 4, the $f$-invariant eigenvalues are $\lambda_{4i}=4i(4i+4n+2)$. The singular orbits for the $U(n+1)$ action are ${\bf CP}^n $ and the homogeneous space $U(n+1)/U(n-1)\times SU(2)$. These have dimensions $2n$ and $4n-3$, respectively. Thus,  $p_f =\frac{n+1}{n-1}>p_{4n}$.

4) Finally, we compute the results for the Ozeki-Takeuchi polynomial defined above. Recall that this is a polynomial $F:\mathbb{H}^{n+1}\times \mathbb{H}^{n+1}\rightarrow \mathbb{R}$ that determines non-homogeneous isoparametric hypersurfaces with four distinct principal curvatures on the sphere $\mathbb{S}^{4(2n+1)+3}$. It is well known that the multiplicities of the principal curvatures of the level sets of an isoparametric function (denoted by $m_i$) further satisfy that $m_1=m_3$ and $m_2=m_4$. In the case of the Ozeki-Takeuchi polynomial  we have that $m_1=3$ and $m_2=4n$. Finally, there exist two focal submanifolds which have codimensions $m_1+1$ and $m_2+1$, respectively. The OT-polynomial is readily seen to be invariant under right multiplication by $Sp(1)$, so it defines a non-homogeneous isoparametric function $f$ on ${\bf HP}^{2n+1}$. In light of the above, the focal submanifolds of this function have dimensions $8n$ and $4n+3$, respectively (and hence $d=4n+3$). Putting everything together we get that the $f$-invariant eigenvalues of $-\Delta_{h_{FS}}$ are $\lambda_{4i} = 4i(4i+4n+2)$ and $p_f=\frac{4n+3}{4n-1}>p_{4(2n+1)}$.

\section{Appendix: Homogeneous metrics on spheres}

We briefly recall some of the properties of homogeneous metrics on spheres and their quotients. These were first studied in \cite{jensen1973,ziller1982}, where the authors find various examples of non-trivial Einstein metrics of this type. These homogeneous metrics can be understood in two different ways. The first one is by looking at groups that act transitively on spheres; homogeneous metrics correspond to decompositions of the isotropy representation into irreducible subrepresentations. The other interpretation is to consider the various Hopf fibrations. Homogeneous metrics arise from viewing spheres of particular dimensions as sphere bundles over certain bases and rescaling the round metric along directions tangent to the Hopf fibers. In what follows we review the algebraic interpretation first and then describe the corresponding Hopf fibration.\\

Groups that act transitively on spheres have been classified in \cite{MS1943,borel1949,borel1950}.  In the cases where the isotropy group acts irreducibly, the resulting homogeneous metrics turn out to be multiples of the round metric, and have already been discussed. The only homogeneous metrics that do not have constant sectional curvature correspond to those where the isotropy representation splits into further irreducible subrepresentations. We will focus on those cases in this section. The homogeneous metrics with reducible isotropy turn out to be invariant under $U(n+1)$, $Sp(n+1)$, or $Spin(9)$ (although their isometry group may be strictly larger). \\

Consider a compact Lie group $G$ acting on a sphere of dimension $N$ with isotropy $H$ so that $\mathbb{S}^{N}=G/H$. The Lie algebra $\mathfrak{g}$ of $G$ has an $\Ad(H)$-invariant decomposition 
$$ \mathfrak{g}=\mathfrak{h} \oplus \mathfrak{p},$$
where $\mathfrak{h}$ is the Lie algebra of $H$ and $\mathfrak{p}$ is a complement of $\mathfrak{h}$. The isotropy group $H$ then acts on $\mathfrak{p}$ by the restriction of the adjoint represenation $\Ad : H \rightarrow GL(\mathfrak{p})$; this corresponds to the isotropy represenation under the identification $\mathfrak{p}\cong T_{[gH]}G/H $ for $g \in G$ (where we identify vectors in $\mathfrak{p}$ with the Killing fields they generate). If we fix a background bi-invariant metric $Q$ on $\mathfrak{g}$, we can further assume that both summands above are $Q$-orthogonal. It is easy to see that $G$-invariant metrics on $\mathbb{S}^{N}$ can be identified with $\Ad (H)$-invariant inner products on $\mathfrak{p}$. The homogeneous metrics on $\mathbb{S}^N$ arise from further decomposing $\mathfrak{p}$ into irreducible subrepresentations. We proceed to study these subrepresentations for the groups mentioned above. 

\subsection{$U(n+1)$-invariant metrics}
The group $G=U(n+1)$ acts transitively on the sphere $\mathbb{S}^{2n+1}$ with isotropy $H=U(n)$. The isotropy representation then splits into two irreducible subrepresentations on $\mathfrak{p}=\mathfrak{p}_0 \oplus \mathfrak{p}_1$, where $\dim \mathfrak{p}_0=1$ and $ \dim \mathfrak{p}_1=2n$. Up to rescaling, any $U(n+1)$-invariant metric is of the form 
$$ g_x= x\; \sigma^2 \oplus  Q\vert_{\mathfrak{p}_1},$$
where $\sigma$ is the standard one-form on $\mathfrak{p}_0 \cong \mathbb{R}$ and $x >0$. We furthermore take an orthogonal basis $X,Y_1,\ldots , Y_n$ of $\mathfrak{p}$ where $X \in \mathfrak{p}_0$, $\| X \|_{g_x}^2=x$, and $(Y_\alpha)$ is an orthonormal basis of $\mathfrak{p}_1\cong \mathbb{C}^{n}$. The corresponding sectional curvatures are then given by
$$ K(X, Y_\alpha )= K(X,i Y_\alpha)= x  $$
$$ K(Y_\alpha, Y_\beta ) = K(Y_\alpha, iY_{\beta})=1 \; \textnormal{ for } \alpha \neq \beta$$
$$  K(Y_\alpha, Y_{\alpha i }) = 4-3x. $$
The scalar curvature of $g_x$ is 
$$ s_{g_x} = 4n^2 + 4n-2nx.$$

These metrics correspond to rescaling the metric by a factor $x$ in the direction tangent to the fiber  in the complex Hopf fibration 

$$\mathbb{S}^1 \rightarrow \mathbb{S}^{2n+1}\rightarrow \bf{CP}^n .$$

\begin{remark}
The metrics $\texttt{g}_t$ in \cite{BettiolPiccione} are obtained by setting $x=t^2$. These metrics exhibit geometric collapse as $t\rightarrow 0$, that is, they converge with bounded curvature in the Gromov-Hausdorff sense to a lower dimensional manifold.
\end{remark}
\subsection{$Sp(n+1)$-invariant metrics} 
%Need to take the t_i's to be different
The group $G=Sp(n+1)$ acts transitively on the sphere $\mathbb{S}^{4n+3}$ with isotropy $H=Sp(n)$. The resulting isotropy representation can be further decomposed into two summands $\mathfrak{p}=\mathfrak{p}_0\oplus \mathfrak{p}_1$ where $\dim \mathfrak{p}_0=3$, $\dim \mathfrak{p}_1=4n$, and such that $H$ acts trivially on $\mathfrak{p}_0$ while $H$ acts irreducibly on $\mathfrak{p}_1$. Thus, we have that $\mathfrak{p}_0 \cong \mathbb{R}^3$ and $\mathfrak{p}_1 \cong \mathbb{H}^n$ as vector spaces. This splitting of $\mathfrak{p}$ corresponds naturally to the decomposition of $T_{[gH]} G/H $ into subspaces that are tangent or normal to the Hopf fiber on the corresponding Hopf fibration. This time the fibers have dimension three, so a $G$-invariant metric $g_x$ on $\mathbb{S}^{4n+3}$ depends on the parameters $x=(x_1,x_2,x_3)$ and is given by
$$ g_x = x_1 (\sigma^1)^2 \oplus x_2 (\sigma^2)^2 \oplus x_3 (\sigma^3)^2 \oplus  Q\vert_{\mathfrak{p}_1},$$
where $\sigma^i$ denotes the standard orthonormal coframe on $ \mathfrak{p}_0 \cong \mathbb{R}^3$ and $x_i>0$. It is worth pointing out that if $Q_0:=\Sigma x_i (\sigma^i)^2$ is a multiple of the Euclidean metric  (i.e. $x_1=x_2=x_3$) then $g_x$ actually has isometry group $Sp(n+1) \times Sp(1)$. The one parameter family of metrics $\texttt{h}_t$ in \cite{BettiolPiccione} are of the this type and are obtained by taking $x_i=t^2$. Equivalently, this metric can be thought of as multiplying along the Hopf fiber by $t^2$ on the fibration 

$$\mathbb{S}^3 \rightarrow \mathbb{S}^{4n+3} \rightarrow {\bf HP}^n .$$

Analogously, if $x_2=x_3$ the resulting metric $g_x$ is $Sp(n+1) \times U(1)$-invariant.\\

In order to describe the curvature of these homogeneous metrics we take an orthogonal basis $X_i$, $1 \leq i \leq 3$ of $\mathfrak{p}_0$ such that $\| X_i \|_{g_x}^2 =x_i$, and an orthonormal basis $Y_\alpha$, $1 \leq \alpha \leq n$, of $\mathfrak{p}_1$. We adopt the notation in \cite{ziller1982} and set $Y_{\alpha 1}=iY_{\alpha}$, $Y_{\alpha 2}=jY_{\alpha}$, and $Y_{\alpha 3}=kY_{\alpha}$.  Also note that $\bar{g}(X_i,Y_\alpha)=0$ since the summands in the splitting $\mathfrak{p}=\mathfrak{p}_0 \oplus \mathfrak{p}_1$ are orthogonal. A standard computation shows that the sectional curvatures (assuming $i,j,k$ is a cyclic permutation of $1,2,3$) of $(\mathbb{S}^{4n+3},\bar{g})$ are given by
$$ K(X_i, X_j )= \frac{\left( -3x_k^2+2x_i x_k + 2 x_j x_k +(x_j - x_i)^2 \right)}{x_1 x_2 x_3}  $$
$$ K(X_i, Y_\alpha ) = x_i$$
$$ K(Y_\alpha, Y_\beta) = K(Y_\alpha, Y_{\beta i}) = 1 \; \textnormal{ for } \alpha \neq \beta$$
$$  K(Y_\alpha, Y_{\alpha i }) = 4-3x_i. $$
From this it is easy to see that the scalar curvature of $\bar{g}$ is given by
$$ s_{g_x}= \frac{2}{x_1 x_2 x_3}\left(x_1^2 +x_2^2 +x_3^2-(x_2-x_3)^2-(x_3-x_1)^2-(x_1-x_2)^2\right)-4n(x_1+x_2+x_3)+16n^2+32n.$$
As we said before, setting $x_i=t^2$ yields the metric $\texttt{h}_t$. Its scalar curvature is
$$ s_{\texttt{h}_t}=\frac{6}{t^2}-12nt^2+16n^2+32n. $$
\begin{remark}
Note that some sectional curvatures blow up to infinity as $t\rightarrow 0$. This contrasts with the $U(n+1)$ invariant metrics above which exhibit geometric collapse as $t\rightarrow 0$, that is, contracting the metric along the Hopf fiber collapses $\mathbb{S}^{2n+1}$ to ${\bf CP}^{n}$ with bounded curvature.
\end{remark}

\subsection{$Spin(9)$-invariant metrics}
Finally, the group $Spin(9)$ acts on transitively on the sphere $\mathbb{S}^{15}$ with isotropy $Spin(7)$. We get an orthogonal decomposition $\mathfrak{g}=\mathfrak{h}\oplus \mathfrak{p}$, where $\mathfrak{g}=\mathfrak{so}(9)$ and $\mathfrak{h}=\mathfrak{so}(7)$. The isotropy representation splits into two irreducible subrepresentations on $\mathfrak{p}=\mathfrak{p}_1 \oplus \mathfrak{p}_2$, where $\dim \mathfrak{p}_1=7$ and $\dim \mathfrak{p}_2=8$. Up to rescaling, $Spin(9)$-invariant metrics on $\mathbb{S}^{15}$ are of the form
$$ g_x=x\, Q\vert_{\mathfrak{p}_1}\oplus Q\vert_{\mathfrak{p}_2}.$$
As before, we set take an orthogonal basis $X_i$ of $\mathfrak{p}_1$ with $\| X_i \|_{\bar{g}}^2=x$ and $Y_\alpha $ an orthonormal basis of $\mathfrak{p}_2$, the sectional curvatures of $\bar{g}$ are given by
$$ K(X_i,X_j)= \frac{1}{x},$$
$$ K(Y_\alpha, Y_\beta)= x,$$
$$ K(X_i, Y_\alpha)= 4-3x.$$
The scalar curvature is then given by
$$ s_{g_x}= \frac{42}{x}-56x+224.$$
The corresponding Hopf fibration for these metrics is $\mathbb{S}^7 \rightarrow \mathbb{S}^{15} \rightarrow \mathbb{S}^8$.

\end{document}